\documentclass[11pt,oneside,reqno]{amsart}
\usepackage{amsmath, amsthm, amscd, amsfonts, amssymb, graphicx, xcolor}
\usepackage[bookmarksnumbered, plainpages]{hyperref}
\boldmath
\definecolor{q0}{RGB}{240,240,240}  
\definecolor{q1}{RGB}{0,204,204}       
\definecolor{q2}{RGB}{0,0,0}       

\newcommand{\PP}{\mathbb{P}}
\newcommand{\QQ}{\mathbb{Q}}

\newcommand{\Z}{\mathbb{Z}}

\newcommand{\LL}{\mathbf{L}}

\newcommand{\CC}{\mathbb{C}}

\newcommand{\ZZ}{\mathbb{Z}}

\newcommand{\HH}{\mathscr{H}}
\newcommand{\FF}{\mathbb{F}}

\newcommand{\BB}{\mathscr{B}}

\newcommand{\Mn}{\mathcal{M}}

\newcommand{\Hh}{\mathscr{H}}

\newcommand{\Mm}{\mathcal{M}}

\newcommand{\mgnbar}{\overline{\Mm}_{g,n}}

\newcommand{\infwedge}{\bigwedge\nolimits^{\frac \infty2}}

\DeclareMathOperator{\Sym}{Sym}

\newcommand{\Aut}{\operatorname{Aut}}

\usepackage{dsfont}
\usepackage{csquotes}
\usepackage{tikz}
\usetikzlibrary{circuits.logic.US,circuits.logic.IEC,fit}

\usetikzlibrary{decorations.pathreplacing,decorations.markings}
\usepackage{tikz-cd}
\usepackage{tkz-euclide,amsmath}
\usepackage{tkz-graph}
\usetikzlibrary{positioning, fit, patterns, arrows.meta}
\tikzset{hidden/.style = {thick, dashed}}
\usetikzlibrary{matrix,arrows, snakes}
\usetikzlibrary{calc}
\usepackage{enumerate}
\usepackage{mathrsfs}
\usepackage{adjustbox} 
\RequirePackage[singlelinecheck=off]{caption}
\usepackage{youngtab}
\usepackage{ytableau}
\usepackage{latexsym}
\usepackage[all]{xy}
\usepackage{theoremref}
\usepackage{mathtools}

\usetikzlibrary{backgrounds}
\usetikzlibrary{calc}
\usetikzlibrary{decorations.markings,arrows}
\usepackage{xstring}

\definecolor{mygreen}{RGB}{44,85,17}
\definecolor{myblue}{RGB}{34,31,217}
\definecolor{mybrown}{RGB}{194,164,113}
\definecolor{myred}{RGB}{255,66,56}

\newcommand*{\myblue}[1]{\textcolor{myblue}{#1}}

\newcommand*{\myred}[1]{\textcolor{myred}{#1}}

\newcommand\DistTo{\xrightarrow{
   \,\smash{\raisebox{-0.2ex}{\ensuremath{\scriptstyle\sim}}}\,}}
\usetikzlibrary{intersections}
\usepackage{pgfplots}
\pgfplotsset{compat=1.10}
\usetikzlibrary{fit,arrows.meta}
\usetikzlibrary{positioning,knots}

\AtBeginDocument{%
   \def\MR#1{}
}

\begin{document}

\title{Formulae for calculating  Hurwitz numbers} 
\author[J.~Ongaro]{Jared Ongaro}  
\address{
School of Mathematics\\
University of Nairobi,\\
00100-Nairobi, Kenya}
\email{ongaro@uonbi.ac.ke}

\begin{abstract}
\vspace{5mm}
In this paper, we aim to provide an accessible survey to various formulae for calculating single Hurwitz numbers. Single Hurwitz numbers count certain classes of meromorphic functions on complex algebraic curves and have a rich geometric structure behind them which has attracted many mathematicians and physicists. Formulation of the enumeration problem is purely of topological nature, but with connections to several modern areas of mathematics and physics.

\end{abstract}


\maketitle




\theoremstyle{definition}
\newtheorem{theorem}{Theorem}[section]
\newtheorem{lemma}[theorem]{Lemma}
\newtheorem{proposition}[theorem]{Proposition}
\newtheorem{corollary}[theorem]{Corollary}
\newtheorem{definition}[theorem]{Definition}
\newtheorem{example}[theorem]{Example}
\newtheorem{xca}[theorem]{Exercise}
\newtheorem{problem}[theorem]{Problem}
\theoremstyle{remark}
\newtheorem{remark}[theorem]{Remark}
\numberwithin{equation}{section}



\section{Introduction}

The number of   non-equivalent  branched coverings  with  a given set of branch points and branched profile  is called the {\bf{Hurwitz number}}. The question  of  determining  the Hurwitz number for a given branch profile is called the {\bf Hurwitz enumeration problem}.
 Hurwitz numbers count the  branched coverings  between complex projective curves with specified branch profile.  Branched coverings were first described  in the famous paper \cite{BR57} by Riemann who developed the idea of  representing nonsingular curves  as branched coverings of $\PP^1$ in order to study their moduli. However,  systematic investigation of branched coverings  was initiated by Hurwitz in \cite{Hur91, Hur02} more than thirty years later. 
 Hurwitz numbers can be computed explicitly for non-complicated branched profiles  thanks to the nice combinatorial interpretations they posses that  they can be interpreted in terms of factorization permutations as first observed by Adolf Hurwitz in \cite{Hur91, Hur02}. \\

Hurwitz numbers have a rich geometric structure behind them which has attracted  many mathematicians and physicists. The effect is that  formulae  for computing Hurwitz numbers arise from different branches of mathematics starting from algebraic geometry, combinatorics,  representation of symmetric groups, tropical geometry.  Explicit answers to the {\bf Hurwitz enumeration problem} are usually difficult to obtain.  One important  case when  this problem has a rather explicit answer, is when at most one branch point  has an arbitrary branch type  while all the others are simple.  In case of $Y=\PP^1$,  we usually suppose that the degenerate branch point is at $\infty\in\PP^1$  and we call its preimages $f^{-1}(\infty)$ poles.  In other words, we are in the situation where all the branch points in $\CC$ are simple,  i.e.  correspond to transpositions  while the permutation at infinity can be described by some partition $\mu=(\mu_1, \ldots,\mu_n)$.\\[2mm]

Hurwitz enumeration problem is an old  but still  active research problem due to its connections  to several  modern areas of mathematics and physics.
Below we walk you through the journey from Topology to Physics in presenting various formulae and connections in calculating single Hurwitz number.
\subsection*{Acknowledgment}
 This paper was written during my fellowship at Department of Mathematics, University of stockholm with funding from International Science Programme, Uppsala University, Sweden. My advisor  Boris Shapiro always  has been invaluable after introducing me to  the subject. I also want to thank  S.~Shadrin,  J.~Bergstr\"om, O.~Bergvall  and  Bal\'azs Szendr\'oi  for being available  for many discussions. Many thanks to M.~Shapiro for explaining nagging details of the ELSV formula.

\section{Partitions and Irreducible representations}
We work over the field of complex numbers. The cardinality of a set $S$ will   be  denoted by $|S|$. All definitions and results on the symmetric group represented below are classical, and can be found in most standard texts.
\begin{definition}
A {\bf{partition}} of a positive integer $d$, is a finite, weakly decreasing sequence  of positive integers $\mu=(\mu_1,\mu_2, \ldots, \mu_n)$ called {\bf{parts}} of $\mu$ such that $\mu_1+\mu_2+\ldots +\mu_n=d.$
\end{definition}
Denote a partition  by $\mu\vdash d$ and  refer to $d$  as the {\bf size} of $\mu$. The number $n$ of parts of $\mu$ is called {\bf length} of $\mu$ and is denoted by  $\ell(\mu).$ 
\begin{example}
There are $5$ integer partitions of $d=4$, namely  $$(4),\; (3\;, 1),\; (2\;, 2),\; (2\;, 1\; ,1),\; (1\;, 1\;, 1\;, 1).$$
\end{example}
Denote the set consisting of the first $d$ positive integers $\{1,2,\ldots, d\}$ by $[d].$ Let $i$ be an integer in the set $\{1,2,\ldots, d\},$ the {\bf multiplicity} of $i$
in $\mu$ which we shall denote by $m_i(\mu)$ is the number of parts $\mu_j$ equaling $i$. We often use exponents to indicate repeated parts, whence a partition $\mu$ can be written multiplicatively as $\mu=1^{m_1(\mu)}\cdot 2^{m_2(\mu)} \ldots k^{m_k(\mu)}$ with $|\mu|=\sum_{i=1}^k i m_i(\mu).$ For instance, the partition $(2\;, 1\; ,1)=1^2\cdot 2.$ 
The number of permutations of the parts of $\mu$  is the quantity 
$$|\Aut(\mu)|=\prod_{i=1}^k m_i(\mu)!\;. $$
We can also represent partitions pictorially using Young diagrams.
\begin{definition}
A { Young diagram} is an array of left and top-justified boxes,
such that the row sizes are weakly decreasing. The Young diagram corresponding to  $\mu=(\mu_1,\mu_2, \ldots, \mu_n)$ is the one that has $n$ rows, and $\mu_i$ boxes in the $i^{th}$ row. 
\end{definition}
For instance, the Young diagrams corresponding to the above mentioned partitions of $4$ are given below.
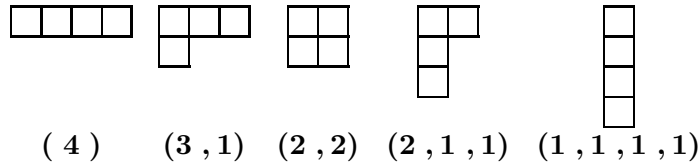
\[
\ytableausetup
{boxsize=1.00em}
\ytableausetup
{aligntableaux=top}
\begin{array}{ccccc}
\Ylinethick{2pt}
\ydiagram{4}\quad & \ydiagram{3,1}\quad & \ydiagram{2,2}\quad &\ydiagram{2,1,1}\quad& \ydiagram{1,1,1,1} \\
     (\;4\;)\quad &  (3\;,1)   \quad    & (2\;, 2)\quad       & (2\;, 1\; ,1)\quad & (1\;, 1\;, 1\;, 1)
\end{array}
\]
The conjugate of the Young tableau $\lambda$ is the reflection of the tableau $\lambda$ along the main diagonal. This is also a standard Young tableau.  We will write $\lambda^t$  to denote the {\bf conjugate} partition of $\lambda$.
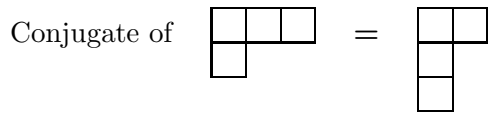
\[
\ytableausetup
{boxsize=1.15em}
\ytableausetup
{aligntableaux=top}
\begin{array}{c}
\Ylinethick{2pt}
 \text{Conjugate of }\quad \ydiagram{3,1}  \quad = \quad\ydiagram {2,1,1} 
\end{array}
\]

Let $S_d$ be the group of all permutations on $[d]$, we make the convention that permutations are multiplied from right to left.  A permutation  $\alpha\in S_d$  is a {\bf cycle of length $k$} or  a $k-${\bf cycle}  if there exist numbers  $i_1,i_2,\ldots, i_k \in [d] $ such that 
$$\alpha(i_1) = i_2, \qquad
\alpha(i_2) = i_3, \qquad
\ldots, \qquad
\alpha(i_k) = i_1.$$
Thus, we can  write $\alpha$ in the form  $(i_1, i_2, \ldots, i_{k}).$  A cycle  of length two is called a {\bf transposition}.  
If we fix $\sigma\in S_d,$ then $\sigma$ can be uniquely decomposed  into a product of disjoint cycles.  The sum of the cycle lengths of $\sigma$ is equal to $d$, so the lenghts form a partition of $d.$  The {\bf cycle type} of $\sigma$ is an expression of the form 
$$1^{m_1}\cdot 2^{m_2} \ldots d^{m_d} ,$$
where the $m_i$ is the number of $i-$cycles in $\sigma.$ We denote the set of all elements conjugate to $\sigma$ in the symmetric group $S_d$ by $C_\sigma,$ that is 
$$C_\sigma=\{\pi\sigma \pi^{-1}:\pi\in S_d \}.$$
Recall that  two permutations are {conjugate} if and only if they have the same cycle type.  Each conjugacy class of $S_d$ corresponds to a partition of $d$ and we can use the combinatorial properties of these partitions to explicitly construct the irreducible representations $S^\lambda$,  from which we can compute the irreducible characters.\\

Young tableaux and symmetric functions  \cite{Mac08} provide not only  a straight-forward way of constructing  irreducible representations of $S_d,$ but also an explicit formula for computing  the corresponding characters. Denote by $\chi^\lambda(C)$ the character of $S^\lambda$ on the conjugacy class $C.$  Since a conjugacy class $C$ of an element in $S_d$ consists of all permutations of the same cycle type, we use the notation $\chi^\lambda_\mu$ to represent the character of $S^\lambda$ at 
the conjugacy class of the cycle type $\mu.$ It can be shown that the dimension of the irreducible representation corresponding to $\lambda$ is given by the {\bf hook formula} $$f^\lambda= \frac{d!}{\prod_{(i,j)\in\lambda }h_{ij}}.$$

The {hooklength} $h_{ij}$ is the number of boxes directly to the right and directly below $(i,j)$ including the box $(i,j)$. In particular,  $h_{ij}=\lambda_i-j+\lambda^t_j-i+1$.  For instance, if $\lambda=(3,1)$, the hook length $h_{(2,1)}$ is $2$.

\begin{example}
The degree of the irreducible representation of $S_4$ corresponding to partition $\lambda=(3\;,1)\vdash 4$  is the number  of  standard tableaux which  can be calculated as 
\[
\ytableausetup
{boxsize=1.55em}
\ytableausetup
{aligntableaux=top}
\Yboxdim{5pt}
 f^{{\tiny\yng(3,1)}}=\dfrac{4!}{4\cdot 2\cdot 1\cdot 1}=3.\]
\end{example}
Let $\mu=(\mu_1, \mu_2, \ldots, \mu_n)\vdash d$ and consider the independent formal variables $x=(x_1, x_2, \ldots, x_m).$ The power sum function $p_\mu(x)$ is defined as 
$$p_\mu(x)=\prod^n_{i=1}(x_1^{\mu_i}+\ldots+ x_m^{\mu_i}).$$
\begin{theorem}[Frobenius Character Formula]
Let $\lambda=(\lambda_1, \lambda_2, \ldots, \lambda_m) $ and the partition $\mu=(\mu_1, \mu_2, \ldots, \mu_n)\vdash d$.
The character $\chi^\lambda_\mu$  is equal to the coefficient of\; $\prod^n_{i=1} x_i^{\lambda_i+m-i}$ in $\Delta(x)p_\mu(x)$
where $\Delta(x)$ is the Vandermonde determinant  $$\prod_{i<j}(x_i-x_j)=\det\left(\begin{matrix} x_1^{n-1} & x_2^{n-1} & \ldots & x_m^{n-1} \\
\vdots & \vdots & \ddots & \vdots \\
x_1 & x_2 & \ldots & x_m \\
1 & 1 & \ldots & 1 \end{matrix}\right).$$
\end{theorem}

Recall that both the conjugacy classes and irreducible representations of $S_{d}$ are in one to one correspondance with partitions of $d$. By $\chi_{\mu}^{\lambda}$ we denote the character of any permutation of cycle type $\mu$ in the representation $\lambda$ of $S_{d}.$

\begin{theorem}[The Burnside formula]
The number of tuples of permutations $(\sigma_{1}, \ldots,\sigma_{m})$ in $S_{d}$ such that:
\begin{enumerate}
\item  $\sigma_{i}$ has cycle type $\mu_{i}$
\item  $\sigma_{1}\ldots\sigma_{m}=\mathds{1}$
\end{enumerate}
is given by the expression
\begin{equation}
\label{burnform}
\sum_{\substack{\lambda \vdash d\\ }} \left( \frac{\dim \lambda}{d!} \right)^2 
\prod_{i=1}^{m}f^\lambda(C_{\mu_i}) 
\end{equation}
where, 
\begin{align*}
f^\lambda(C_{\mu^i})=&\frac{|C_{\mu^i}|}{\dim \lambda} \chi^\lambda(\mu^i). \\[2mm]
\end{align*}
\end{theorem}

In what follows, a curve, always means a smooth complex projective algebraic curve.

\section{Hurwitz Numbers and Hurwitz Spaces}
Let  $X$  be  a complex nonsingular curve of genus $g$ (Note that we impose further conditions on $X$, see for example in \cite{JOplane, JOplanarity} ). A single Hurwitz numbers $h_{g,\mu}$  enumerate Hurwitz covering arising from meromorphic functions  on $X$.   A {\em Hurwitz covering of type} $(g, \mu)$  is a meromorphic function   $f : X\rightarrow \mathbb{C}$ on $X$ with labelled poles $\{p_{1}, \ldots, p_{n}\}$ given by a divisor $\mu_1p_1+\ldots + \mu_np_n$, and  except for these poles, the holomorphic $1$-form  of $f$ has simple zeros on $X\backslash \{p_{i}, \ldots, p_{n}\}$ with distinct critical values of $f$.  A meromorphic function  $f$  gives a finite morphism to the complex projective line $\mathbb P^1$ whose degree $d$ by definition is  the degree of the  morphism  $f: X\longrightarrow  \mathbb P^1$.
Observe, we are call a meromorphic function a covering because if we remove the critical values of $f$ (including $\infty$) from $\mathbb{P}^{1}$, then on this open set $f$ becomes a topological covering. More precisely, let the branched locus $B=\{z_{1}, \ldots, z_{w}, \infty\}$ denote the set of distinct critical values of $f$. Then
$$
f_{0}: X\backslash f^{-1}(B)\rightarrow \mathbb{P}^{1}\backslash B
$$
is a topological covering of degree $d$.

 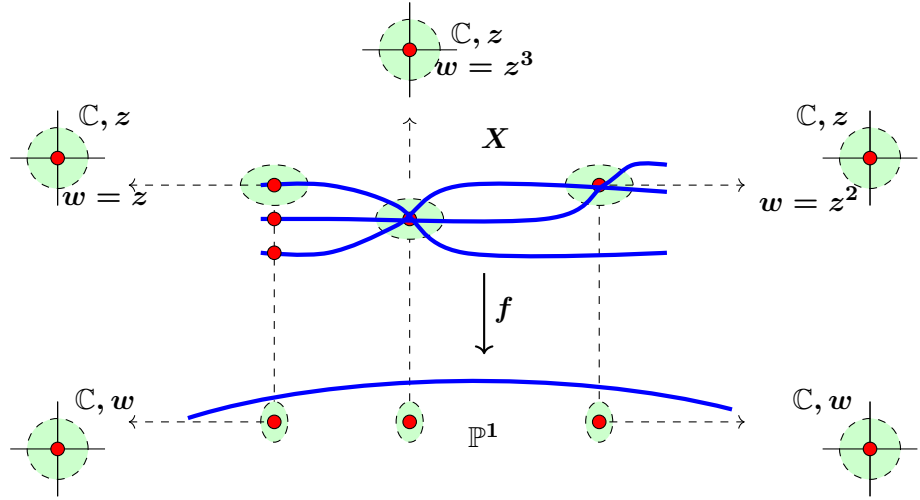
\begin{figure}[htbp]
\centering
\begin{tikzpicture}[scale=.90,mycirc/.style={fill=green!20, dashed},
mydash/.style={ dashed}, blueline/.style={thick, blue},
]
\foreach \x / \y in {-6 / 0.9,-6/-3.4 , 6/0.9, 6/-3.4, -0.8/2.5  }{
\draw [mycirc] (\x, \y) circle(.45cm);
\draw (\x -0.7, \y) --(\x+0.7, \y);
\draw (\x , \y+0.7) --(\x, \y-0.7);
\draw (\x , \y+0.7) --(\x, \y-0.7);
\draw [fill=red] (\x, \y) circle(.1cm);
}
\draw[mydash] (-0.8, 0) --(-0.8,-3);
\draw[mydash] (-2.8,0.6) --(-2.8,-3);
\draw[mydash] (2,0.5) --(2,-3);
\draw[mycirc] (-2.8,-3) ellipse [x radius=.2cm,y radius=.3cm] node (leftupper){};
\draw[mycirc] (-2.8,0.5) ellipse [x radius=.5cm,y radius=.3cm] node (leftlower){};
\draw[mycirc] (-0.8,0) ellipse [x radius=.5cm,y radius=.3cm];
\draw[mycirc] (-0.8,-3) ellipse [x radius=.2cm,y radius=.3cm];
\foreach \y in {-3,0}{
\draw [fill=red] (-0.8, \y) circle(.1cm);}
\draw[mycirc] (2,0.5) ellipse [x radius=.5cm,y radius=.3cm]node (rightupper){};
\draw[mycirc] (2,-3) ellipse [x radius=.2cm,y radius=.3cm]node (rightlower){};
\foreach \y in {-3,0.5}{
\draw [fill=red] (2, \y) circle(.1cm);}
\draw[dashed,->] (leftupper.west) --++(left:2cm);
\draw[dashed,->]  (leftlower.west) --++(left:2cm);
\draw[dashed,->]  (rightupper.east) --++(right:2cm);
\draw[dashed,->]  (rightlower.east) --++(right:2cm);
\draw[dashed,->] (-0.8,0.6) -- (-0.8,1.5);

\node at (-5.3, 1.5) {$\mathbb C, z$};
\node at (-5.3, -2.7) {$\mathbb C,w$};
\node at (5.3, 1.5) {$\mathbb C,z$};
\node at (5.3, -2.7) {$\mathbb C,w$};
\node at (0.2, 2.7) {$\mathbb C, z$};
\node at (0.3,2.3){$w=z^3$};

\node at (-5.3,0.3){$w=z$};
\node at (5.1,0.3){$w=z^2$};
\draw[smooth, blueline, ultra thick] plot[ tension=0.6]coordinates { (-3, 0)(-2,0.0) (1.2,0.0)  (2.07,0.5) (2.5,0.8)(3, 0.8) };
\draw[smooth, blueline, ultra thick] plot[ tension=0.6]coordinates { (-3,0.5)(-2,0.5) (-1,0.2)  (0,-0.5)(3,-0.5) };
\draw[smooth, blueline, ultra thick] plot[ tension=0.6]coordinates { (-3,-0.5)(-2,-0.5) (-1,-0.1)  (0,.5) (3,.4) };
\draw[->, thick] (0.3, -0.8) -- (0.3,-2.0);
\draw (0.3, -1.0) node[below right]{$f$};
\node[right] at (0.1, 1.2) {$X$};
\node at (0.3,-3.2){$\PP^1$};
\draw[color=blue, ultra thick] (-3.5,-3.5) ++(136:0.8cm) arc (136:50:5.9cm and 1.8cm);
\foreach \y in {-3, 0.5, 0, -0.5}{
\draw [fill=red] (-2.8, \y) circle(.1cm);}
\end{tikzpicture}
\caption{Local picture  of a  Hurwitz covering of degree $3.$}
\label{cav10}
\end{figure}

A holomorphic map $f: X\longrightarrow \PP^1$ is called a {\em meromorphic function}. Thus,  given a meromorphic function  $f$,  for $\infty\in \mathbb P^1$ we have the polar divisor $f^{-1}(\infty)=\mu_1p_1+\ldots + \mu_np_n,$  where $p_1, \ldots, p_n$ are distinct points on $X$ and  $\mu=(\mu_{1}, \ldots, \mu_{n})\vdash d$  the {\em branch type} of $f$ at a point $\infty$  gives a partition of $d$.  For instance,  the branch type for a {\em simple} branch point is $(2)$ or $(2, 1, \ldots, 1)$. 

\begin{example}
Let $X$ be the cubic curve in $\PP^2$ defined by $y^2z=x(x+z)(x-z)$, where $[x,y,z]$ are homogenous coordinates in $\PP^2$ as discussed in \cite{JOzeu}.
Let $f$ be the linear projection of $X$ from $p=[0,1,0]\in \PP^2\backslash X$
onto $\PP^1$. It defines a $2$-sheeted  branched covering  of $\PP^1$.   All  the  $4$ branch points   namely;  $$[0\;,\; 1],\; [1\;,\;0],\;[-1\;,\;1] \;\text{and}\; [1\;,\;1]\in \PP^1$$
are simple implying  that  the meromorphic function on the  linear projection of $X$ to $\PP^1$ from a point $p=[0,1,0]$ is a simple branch covering.
\end{example}

The set of all branch points $B$ is called the {\em branching locus} of $f$.   In this way, every non-constant meromorphic function on a curve $X$ is a {\em branched covering}.  In  Hurwitz covering, we consider the case where branch type at $\infty$ is given by the partition $\mu=(\mu_{1}, \ldots, \mu_{n})\vdash d$ and there are exactly $w=2g-2+n+d$ simple branch points by Riemann-Hurwitz formula. The basic problem is then the enumeration of such maps $f: X\longrightarrow \PP^1$ for a given $g$ and $d$ for a  prescribed branch type over each branch point  of $f$.

\begin{definition}
\label{def:eqcover}
Two Hurwitz covering $f_1: X_1\longrightarrow \PP^1, $ covering with poles at $\{p_{1}, \ldots, p_{n}\}$ and $f_{2} : X_{2}\longrightarrow \PP^{1}$ with poles at $\{q_{1}, \ldots, q_{n}\}$  are  called  {\bf equivalent} if there exists an isomorphism  $\phi:X_1\DistTo X_2$ such that $\phi(p_{i})=q_{i},$ for all $i=1, \ldots, n$ and  the diagram

\begin{center}
\begin{tikzcd}[column sep=0.7cm,row sep=0.9cm]
X_1 \arrow[swap]{rr}{\phi}[swap]{\sim}
 \arrow[swap]{dr}{f_1}& &X_2 \arrow{dl}{f_2}\\
& \PP^1& 
\end{tikzcd} 
\end{center}
commutes.
\end{definition}

If $X_{1}=X_{2}, p_{i}=q_{i}$, and $f_{1}=f_{2}=\phi$, then $\phi$ is an {\em automorphism} of a Hurwitz covering $f$. Thus,  we count a Hurwitz covering with the automorphism factor $ \frac{1}{|\Aut(f)|}$ to compensate for the relabeling.   If  $\phi : X_{1}\rightarrow X_{2}$ is  just a homeomorphism, we say $f_{1}$ and $f_{2}$ have the same  {\em topological type}.\\[2mm]

Throughtout we let the integers $d\geq 1$, $g\geq 0$ and a partition $\mu=(\mu_{1}, \ldots, \mu_{n})\vdash d$. 

\subsection{ Topological Definition of Hurwitz Numbers}

\begin{definition}
The connected  single Hurwitz number  is
\begin{equation}
h_{g,\mu}=\sum_{[f]}\; \frac{1}{|\Aut(f)|},
\end{equation}

where the sum runs all topological equivalence classes of Hurwitz coverings
of type $(g, \mu)$ for $g\geq 0$ and a partition $\mu=(\mu_{1}, \ldots, \mu_{n})\vdash d$ for connected complex nonsingular curves $X$ 
of genus $g$.  
\end{definition}

Its convenient to define the disconnected Hurwitz numbers $h_{g,\mu}^{\bullet}$ but without the condition that the covering surface be connected.
as the connected Hurwitz numbers since we can be computed from  the  disconnected Hurwitz via the inclusion-exclusion formula.

\begin{definition}
The disconnected single Hurwitz number of type $(g, \mu)$ for $g\geq 0$ and  a partition $\mu=(\mu_{1}, \ldots, \mu_{n})\vdash d$  is

\begin{equation}
h_{g,\mu}^{\bullet}=\sum_{[f]}\; \frac{1}{|\Aut(f)|},
\end{equation}

where the sum runs all topological equivalence classes of Hurwitz coverings
$f: X\longrightarrow \PP^1$ for possibly a disconnected complex nonsingular curves $X$  of genus $g$.  
\end{definition}

\subsection{ Geometric Formulation of Single Hurwitz Numbers}
Fix $g\geq 0$ and a partition $\mu\vdash d$ on  branched coverings $f: X\to \PP^1$ and the number $w$ of  branch points , then equivalence classes of branched coverings form a moduli space  called {\em single Hurwitz space}  of type $(g, \mu)$ Hurwitz coverings denoted  by 

\begin{equation}
\label{eq:hs}
    \begin{array}{ccc}
       \Hh_{g,\mu}=\left\{\mbox{ $f: X \longrightarrow \PP^1$}\ %
        \bigg |\begin{array}{c}
         \ \mbox{$\mu\vdash d$, $X$ has genus $g$ and  $f$ is a   }   \\[2mm]
             \mbox{ Hurwitz  covering of type $(g, \mu)$  } \\

        \end{array}
                  \right\}
                  \bigg/ \sim. \quad
    \end{array}
\end{equation}
 $\Hh_{g,\mu}$ possess the structure of an irreducible smooth algebraic variety (see  {\textsection 21} of  \cite{ACG11} or  \cite{Ful69})
of dimension equal to $w=2g+2d-2$.  The fundamental group of the configuration space of $w$ branch points in $\PP^1$ acts on the fibers of $\mathscr H_{g,\mu}$  and the orbits of this action are in one to one correspondence with the connected components of  $\mathscr H_{g,\mu}$. Furthermore,  $\Hh_{g, \mu}$  comes with  a natural  finite \'{e}tale  covering   
\begin{equation}
\label{HF1}
\begin{aligned}
\Phi: \Hh_{g, \mu}\longrightarrow & \Sym^w \PP^1\backslash \Delta\\[2mm]
(f:X\longrightarrow\PP^1)\longmapsto & \{\text{branch locus of $f$}\}
\end{aligned}
\end{equation}
where $\Sym^w\PP^1$ is the space of unordered $w-$tuples of points  in $\PP^1$ and $\Delta$ is the discriminant hypersurface corresponding to sets of cardinality less than $w$.  The morphism $\Phi$ is called the {\em branching morphism} and its degree  is  the {\em single Hurwitz number} $h_{g, \mu}={\Phi}^{-1}(B)$ for a fixed branch locus $B$.

\subsection{ Group Theoretic Formulation  of Hurwitz Numbers}
To every  meromorphic function $f:X\longrightarrow \PP^1$  of degree $d$ we can associate its monodromy data and we obtain a formulation of single Hurwitz numbers in terms of counting sequences of factorization of a permutations.

\begin{definition}
Fix  $\sigma\in S_d,$ a sequence  $(a_1b_1),(a_2b_2), \ldots, (a_nb_n) $ such that the product  
$$(a_1b_1)(a_2b_2) \ldots (a_mb_m)=\sigma $$
 is called a {\bf transposition factorization} of $\sigma$ of length $m$.
 \end{definition}
 
The factorization is not unique, for instance  $(123)=(12)(13)=(13)(23)$. However, the number of transpositions in the factorization  depends on the cycle type of  the permutation $\sigma$ rather than the permutation itself. 
\begin{definition}
\label{SH}
Let  $\mu\vdash d$ for  $d\geq 1$. Consider an ordered sequence of permutations $(\tau_1, \ldots, \tau_w, \sigma)\in (S_d)^{w+1}$ where $w=2g-2+n+d$. The single Hurwitz number  $h_{g,\mu}$ is given by
\begin{equation}
\begin{array}{ccc}
h_{g,\mu}= &\frac{1}{d!}\times  \text{number  of ordered $w$-tuples} 
\mbox{ $(\tau_1, \ldots,\tau_w)\in (S_d)^w$}
\end{array}
\end{equation}
 such that:
\begin{enumerate}
\item  $\big( \tau_1, \ldots, \tau_w\big )$ are transpositions  in $S_d$,
\item  the product $\tau_1\circ  \cdots\circ \tau_w=\sigma$ in $S_d$ whose  cycle type  is $\mu$.
\item  The subgroup $\langle \tau_{1}, \ldots,\tau_{w} \rangle\subset S_{d}$ acts transitively on the set $\left \{1, 2, \ldots, d\right \}.$
\end{enumerate}

\end{definition}

Observe that the third condition is equivalent to requiring that  the covering surface be connected.  So we can define the disconnected  Hurwitz numbers by relaxing the third condition:

\begin{definition}
\label{SH}
The disconnected  single  Hurwitz number  is
\begin{equation}
\begin{array}{ccc}
h_{g,\mu}^\bullet= & \text{number  of}\left\{\mbox{ $(\tau_1, \ldots,\tau_w)\in (S_d)^w$}\ %
\bigg |\begin{array}{c}
\ \mbox{$\tau_i$ are transpositions } \mbox{with}  \\
\mbox{$ \tau_1\circ  \cdots\circ \tau_w=\sigma\in S_d$ } \\
\end{array}
\right\}
\bigg/ \mbox{$d!$}

\end{array}\end{equation}
 such that $\sigma$ in $S_d$ has  cycle type  is $\mu$.

\end{definition}

In this form the problem was  for the first time formulated by A. Hurwitz.
\subsection{ The Hurwitz Formula}
\label{THF}
In  several specific cases A. Hurwitz calculated  $h_{g, \mu}$   using purely combinatorial methods in 1891 and in terms of irreducible characters of $S_n$ in 1902.  In  \cite{Hur91}  he also observed  that  the calculation $h_{g, \mu}$  is a purely group-theoretic problem, but its solution is  complicated for arbitrary $g$ and $d$.  On page 17 of  \cite{Hur91}, Hurwitz found  answers for calculating the degree of the map (\ref{HF1}) for small $d\leq 6$ and any $g\geq 0$. Namely, 

\begin{equation}
\begin{aligned}
h_{g, {\;\ytableausetup
{boxsize=1.55em}
\ytableausetup
{aligntableaux=top}
\Yboxdim{5pt}
\tiny\yng(2)}}=&\frac{1}{2},\\[4mm]
h_{g, _{\;\ytableausetup
{boxsize=1.55em}
\ytableausetup
{aligntableaux=top}
\Yboxdim{5pt}
\tiny\yng(2,1)}}=&\frac{1}{3!}(3^{2g+3}-3),\\[4mm]
h_{g, {\;\ytableausetup
{boxsize=1.55em}
\ytableausetup
{aligntableaux=top}
\Yboxdim{5pt}
\tiny\yng(2,1,1)}}=&\frac{1}{4!}(2^{2g+4}-4)(3^{2g+5}-3),\\[5mm]
h_{g, _{\;\ytableausetup
{boxsize=1.55em}
\ytableausetup
{aligntableaux=bottom}
\Yboxdim{5pt}
\tiny\yng(2,1)}}=&\frac{10^{2g+8}}{7200}-\frac{6^{2g+8}}{288} +\frac{5^{2g+8}}{450}-\frac{4^{2g+8}}{72} + \frac{{3^{2g+8}}}{18}+\frac{{2^{2g+8}}}{12}-\frac{5}{9},\\[5mm]
h_{g, _{\;\ytableausetup
{boxsize=1.55em}
\ytableausetup
{aligntableaux=bottom}
\Yboxdim{5pt}
\tiny\yng(2,1,1,1,1)}}=&\frac{15^{2g+10}}{2\cdot (360)^2}-\frac{10^{2g+10}}{7200}+\frac{9^{2g+10}}{2\cdot (72)^2}-\frac{7^{2g+10}}{2\cdot (24)^2}+ \frac{6^{2g+10}}{2\cdot (36)^2}-\frac{5^{2g+10}}{360}+
\\[4mm]
&\qquad \qquad\qquad\qquad\qquad\qquad+\frac{4^{2g+10}}{36} -\frac{19}{324}\cdot 3^{2g+10}-\frac{19}{144}\cdot 2^{2g+10}+\frac{727}{1152}.\\[4mm]
\end{aligned}
\label{Hurfomd}
\end{equation}

For instance,  it is  immediate  to  enumerate all degree $3$ single Hurwitz numbers for  all  $g\geq 0$.\\[2mm]


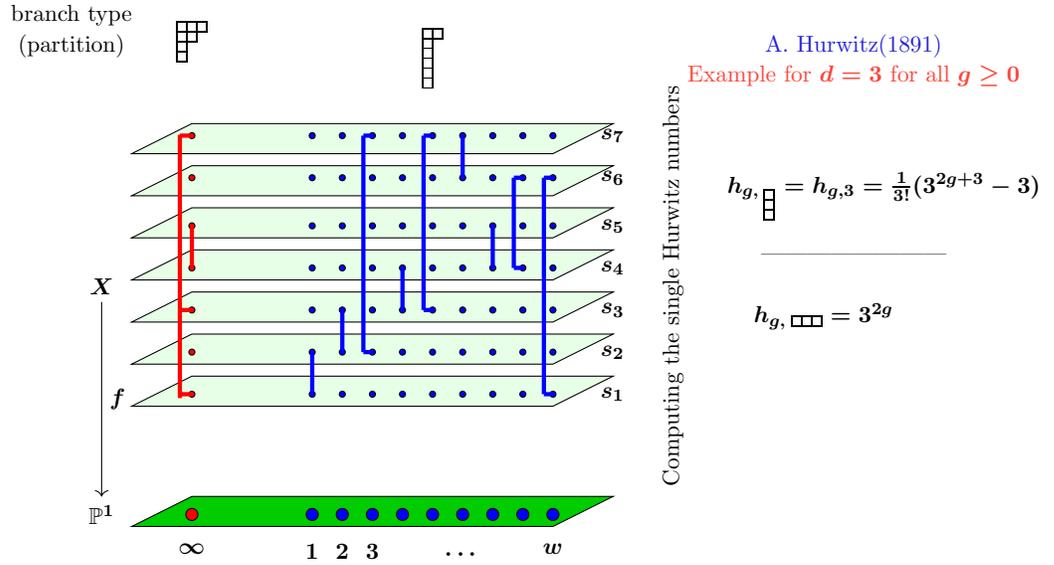
\begin{figure}[htbp]
\begin{center}
\begin{tikzpicture}[scale=0.8, every node/.style={scale=0.8}]
\draw[fill=green!10] (6,4.7) -- (13,4.7) -- (12,4.2) -- (5,4.2)
-- cycle;
\draw [fill=red] (6.0, 4.5) circle(.05cm)node[above=25pt]{$ \ytableausetup
{boxsize=1.55em}
\ytableausetup
{aligntableaux=top}
\Yboxdim{5pt}
\tiny\yng(3,2,1,1)
$};
\draw [fill=blue] (8.0, 4.5) circle(.05cm);
\draw [fill=blue] (8.5, 4.5) circle(.05cm);
\draw [fill=blue] (9.0, 4.5) circle(.05cm);
\draw [fill=blue] (9.5, 4.5) circle(.05cm);
\draw [fill=blue] (10.0, 4.5) circle(.05cm)node[above=15pt]{$
\ytableausetup
{boxsize=1.55em}
\ytableausetup
{aligntableaux=top}
\Yboxdim{5pt}
\tiny\yng(2,1,1,1,1,1)
$};;
\draw [fill=blue] (10.5, 4.5) circle(.05cm);
\draw [fill=blue] (11.0, 4.5) circle(.05cm);
\draw [fill=blue] (11.5, 4.5) circle(.05cm);
\draw [fill=blue] (12.0, 4.5) circle(.05cm)node[right=15pt]{$s_7$};
\draw[fill=green!10] (6,4.0) -- (13,4.0) -- (12,3.5) -- (5,3.5)
-- cycle;
\draw [fill=red] (6.0, 3.8) circle(.05cm);
\draw [fill=blue] (8.0, 3.8) circle(.05cm);
\draw [fill=blue] (8.5, 3.8) circle(.05cm);
\draw [fill=blue] (9.0,3.8) circle(.05cm);
\draw [fill=blue] (9.5, 3.8) circle(.05cm);
\draw [fill=blue] (10.0, 3.8) circle(.05cm);
\draw [fill=blue] (10.5, 3.8) circle(.05cm);
\draw [fill=blue] (11.0, 3.8) circle(.05cm);
\draw [fill=blue] (11.5,3.8) circle(.05cm);
\draw [fill=blue] (12.0, 3.8) circle(.05cm)node[right=15pt]{$s_6$};

\draw[fill=green!10] (6,3.3) -- (13,3.3) -- (12,2.8) -- (5,2.8)
-- cycle;
\draw [fill=red] (6.0, 3.0) circle(.05cm);
\draw [fill=blue] (8.0, 3.0) circle(.05cm);
\draw [fill=blue] (8.5, 3.0) circle(.05cm);
\draw [fill=blue] (9.0,3.0) circle(.05cm);
\draw [fill=blue] (9.5, 3.0) circle(.05cm);
\draw [fill=blue] (10.0, 3.0) circle(.05cm);
\draw [fill=blue] (10.5, 3.0) circle(.05cm);
\draw [fill=blue] (11.0, 3.0) circle(.05cm);
\draw [fill=blue] (11.5,3.0) circle(.05cm);
\draw [fill=blue] (12.0, 3.0) circle(.05cm)node[right=15pt]{$s_5$};
\draw[fill=green!10] (6,2.6) -- (13,2.6) -- (12,2.1) -- (5,2.1)
-- cycle;
\draw [fill=red] (6.0, 2.3) circle(.05cm);
\draw [fill=blue] (8.0, 2.3) circle(.05cm);
\draw [fill=blue] (8.5, 2.3) circle(.05cm);
\draw [fill=blue] (9.0,2.3) circle(.05cm);
\draw [fill=blue] (9.5, 2.3) circle(.05cm);
\draw [fill=blue] (10.0, 2.3) circle(.05cm);
\draw [fill=blue] (10.5, 2.3) circle(.05cm);
\draw [fill=blue] (11.0, 2.3) circle(.05cm);
\draw [fill=blue] (11.5,2.3) circle(.05cm);
\draw [fill=blue] (12.0,2.3) circle(.05cm)node[right=15pt]{$s_4$};
\draw[fill=green!10] (6,1.9) -- (13,1.9) -- (12,1.4) -- (5,1.4)
-- cycle;
\draw [fill=red] (6.0, 1.6) circle(.05cm);
\draw [fill=blue] (8.0, 1.6) circle(.05cm);
\draw [fill=blue] (8.5, 1.6) circle(.05cm);
\draw [fill=blue] (9.0,1.6) circle(.05cm);
\draw [fill=blue] (9.5, 1.6) circle(.05cm);
\draw [fill=blue] (10.0, 1.6) circle(.05cm);
\draw [fill=blue] (10.5, 1.6) circle(.05cm);
\draw [fill=blue] (11.0, 1.6) circle(.05cm);
\draw [fill=blue] (11.5,1.6) circle(.05cm);
\draw [fill=blue] (12.0,1.6) circle(.05cm)node[right=15pt]{$s_3$};
\draw[fill=green!10] (6,1.2) -- (13,1.2) -- (12,0.7) -- (5,0.7)
-- cycle;
\draw [fill=red] (6.0, 0.9) circle(.05cm);
\draw [fill=blue] (8.0, 0.9) circle(.05cm);
\draw [fill=blue] (8.5, 0.9) circle(.05cm);
\draw [fill=blue] (9.0,0.9) circle(.05cm);
\draw [fill=blue] (9.5, 0.9) circle(.05cm);
\draw [fill=blue] (10.0, 0.9) circle(.05cm);
\draw [fill=blue] (10.5, 0.9) circle(.05cm);
\draw [fill=blue] (11.0, 0.9) circle(.05cm);
\draw [fill=blue] (11.5,0.9) circle(.05cm);
\draw [fill=blue] (12.0,0.9) circle(.05cm)node[right=15pt]{$s_2$};

\draw[fill=green!10] (6,0.5) -- (13,0.5) -- (12,0.0) -- (5,0.0)
-- cycle;
\draw [fill=red] (6.0, 0.2) circle(.05cm);
\draw [fill=blue] (8.0, 0.2) circle(.05cm);
\draw [fill=blue] (8.5, 0.2) circle(.05cm);
\draw [fill=blue] (9.0,0.2) circle(.05cm);
\draw [fill=blue] (9.5, 0.2) circle(.05cm);
\draw [fill=blue] (10.0, 0.2) circle(.05cm);
\draw [fill=blue] (10.5, 0.2) circle(.05cm);
\draw [fill=blue] (11.0, 0.2) circle(.05cm);
\draw [fill=blue] (11.5,0.2) circle(.05cm);
\draw [fill=blue] (12.0,0.2) circle(.05cm)node[right=15pt]{$s_1$};
\draw[fill=green!80!black] (6,-1.5) -- (13,-1.5) -- (12,-2.0) -- (5,-2.0)
-- cycle;
\draw [fill=red] (6.0, -1.8) circle(.1cm) node[below=8pt]{$\infty$};
\draw [fill=blue] (8.0, -1.8) circle(.1cm)node[below=8pt]{$1$};
\draw [fill=blue] (8.5, -1.8) circle(.1cm)node[below=8pt]{$2$};
\draw [fill=blue] (9.0,-1.8) circle(.1cm)node[below=8pt]{$3$};
\draw [fill=blue] (9.5, -1.8) circle(.1cm);
\draw [fill=blue] (10.0, -1.8) circle(.1cm);
\draw [fill=blue] (10.5, -1.8) circle(.1cm)node[below=10pt]{$\cdots$};
\draw [fill=blue] (11.0, -1.8) circle(.1cm);
\draw [fill=blue] (11.5,-1.8) circle(.1cm) node[right=35pt]{$$};
\draw [fill=blue] (12.0,-1.8) circle(.1cm)node[below=8pt]{$w$};
\node (h) at (5.8, 4.65){};
\node (j) at (5.8, 0.0){};
\draw[ultra thick, -, color=red] node[above]{} (h) -- node[right] {} (j);
\draw[ultra thick, -, color=red]  (5.8, 4.5) --(6.0, 4.5);  
\draw[ultra thick, -, color=red]  (5.8, 1.6) --(6.0, 1.6);
\draw[ultra thick, -, color=red]  (5.8, 0.2) --(6.0, 0.2);
\draw[ultra thick, -, color=red] node[above]{} (6.0, 3.0) -- node[right] {} (6.0, 2.3); 
\draw[ultra thick, -, color=blue] node[above]{} (8.0, 0.2) -- node[right] {} (8.0, 0.9); 
\draw[ultra thick, -, color=blue] node[above]{} (8.5, 0.9) -- node[right] {} (8.5, 1.6); 
\draw[ultra thick, -, color=blue] node[above]{} (8.85, 0.9) -- node[right] {} (8.85, 4.5);
\draw[ultra thick, -, color=blue]  (8.85, 0.9) --(9.0, 0.9);
\draw[ultra thick, -, color=blue]  (8.85, 4.5) --(9.0, 4.5);
\draw[ultra thick, -, color=blue] node[above]{} (9.5, 1.6) -- node[right] {} (9.5, 2.3); 
\draw[ultra thick, -, color=blue] node[above]{} (9.85, 1.6) -- node[right] {} (9.85, 4.5);
\draw[ultra thick, -, color=blue]  (9.85, 1.6) --(10.0, 1.6);
\draw[ultra thick, -, color=blue]  (9.85, 4.5) --(10.0, 4.5);
\draw[ultra thick, -, color=blue] node[above]{} (10.5, 3.8) -- node[right] {} (10.5, 4.5); 
\draw[ultra thick, -, color=blue] node[above]{} (11.0, 2.3) -- node[right] {} (11.0, 3.0); 
\draw[ultra thick, -, color=blue] node[above]{} (11.35, 2.3) -- node[right] {} (11.35, 3.8);
\draw[ultra thick, -, color=blue]  (11.35, 2.3) --(11.5, 2.3);
\draw[ultra thick, -, color=blue]  (11.35, 3.8) --(11.5, 3.8);
\draw[ultra thick, -, color=blue] node[above]{} (11.85, 0.2) -- node[right] {} (11.85, 3.8);
\draw[ultra thick, -, color=blue]  (11.85, 0.2) --(12.0, 0.2);
\draw[ultra thick, -, color=blue]  (11.85, 3.8) --(12.0, 3.8);
\node (A) at (4.5, 2) {$X$};
\node (B) at (4.5,-1.8) {$\PP^1$};
\draw[->] node[above]{} (A) -- node[right] {$f$} (B);
\node () at (4, 6.5) {branch type};
\node () at (4, 6.0) {(partition)};
\node[rotate=90] () at (14, 2.0) { Computing the single Hurwitz numbers};
\node () at (17, 6.0) {\myblue{A. Hurwitz(1891)}};
\node () at (17, 5.5) {\myred{Example for $d=3$ for all $g\geq 0$}};
\node () at (17.5, 3.5) {$h_{g,_{\;\ytableausetup
{boxsize=1.55em}
\ytableausetup
{aligntableaux=top}
\Yboxdim{5pt}
\tiny\yng(1,1,1)}}=h_{g, 3}=\frac{1}{3!}(3^{2g+3}-3)$};
\node () at (17, 2.5) {------------------------};
\node () at (16.5, 1.5) { $h_{g, \; {\ytableausetup
{boxsize=1.55em}
\ytableausetup
{aligntableaux=top}
\Yboxdim{5pt}
\tiny\yng(3)}}=3^{2g}$ };
\end{tikzpicture}
\end{center}
\caption{Associated to counting permutations up to conjugation in $S_d$ }

\end{figure}


\begin{example}
Let $\mu=(2,1)\vdash 3$ and $g\geq 0$. To compute, 
$h_{g, _{\;\ytableausetup
{boxsize=1.55em}
\ytableausetup
{aligntableaux=top}
\Yboxdim{5pt}
\tiny\yng(2,1)}}$, 
all we need, is  to count  sequences of  $w=2g+4$  transpositions with the above properties. That is, we  need to  count sequences of $w$ transpositions which generate a transitive subgroup of $S_d$ whose product is identity.   Notice that we are free to choose $2g+3$  elements of the sequence  as the last  of them is determined  by the requirement that the product  must be  identity as the product  of $2g+3$ transpositions as the same parity as one transposition in $S_3$.  Also,  to avoid disconnected coverings we have to avoid always choosing the  same transpositions $2g+3$ times.  Thus, we immediately find the number of simple  branched coverings of degree $3$ is
$$h_{g, _{\;\ytableausetup
{boxsize=1.55em}
\ytableausetup
{aligntableaux=top}
\Yboxdim{5pt}
\tiny\yng(2,1)}}=\frac{3^{2g+3}-3}{6}$$
for all $g\geq 0$ as found by A. Hurwitz in Equation \eqref{Hurfomd}
\end{example}

Similary, $h_{g, _{\;\ytableausetup
{boxsize=1.55em}
\ytableausetup
{aligntableaux=top}
\Yboxdim{5pt}
\tiny\yng(3)}}$ the  number for non-isomorphic branched coverings of degree $3$ over $\PP^1$ with one complicated branch point  can easily be calculated.

\begin{example}
Indeed we establish that the single Hurwitz number $h_{g, _{\;\ytableausetup
{boxsize=1.55em}
\ytableausetup
{aligntableaux=top}
\Yboxdim{5pt}
\tiny\yng(3)}}=3^{2g}$ as  follows. Notice that for complicated branch point we can choose freely  any $3$-cycle in $S_3$.  The $3$-cycle  guarantee that  we generate  $S_3$. Then we are free to choose cycle for the next $2g+1$ simple branch points, the last is uniquely determined by the fact that the multiplication is identity.  So we get $2\cdot 3^{2g+1}$ elements of $S_3$. We divide by $3!$  to account for  relabelling of  the sheets of the branched coverings. 
\end{example}

\subsection{ Minimal Transposition Factorisation}
For genus $g=0$, the single Hurwitz number $h_{0, \mu}$ is  equivalent  to counting factorizations of a permutation $\sigma\in S_d$ of  cycle type $\mu\vdash d$ into a product of transpositions of minimal length divided by $d!\;$,  a result known  and published by Hurwitz.
\begin{definition}
Let $\sigma\in S_d$ be a fixed permutation of length $m$. The  sequence $(\tau_1 , \ldots , \tau_n)$ is called a {\bf minimal transitive factorization}  of $\sigma$ into transpositions if the following $3$ conditions are satisfied:
\begin{enumerate}
\item {\bf Product cycle type condition:} $\tau_1  \ldots  \tau_n=\sigma,$
\item  {\bf Minimality condition:}  $ n:=m+d-2$,
\item   {\bf Transitivity condition:} The graph $G_\sigma$ is connected, where $G_\mu$ is the graph corresponding to factorization $\sigma$ into  a product of $n$ transpositions.

\end{enumerate}

\end{definition}

Note that, one needs   at least $d-1$ transpositions  to build a cycle of length $d$. Then $n\geq d-1$.
\begin{example}
\leavevmode
\begin{enumerate}
\item If  $\mu=(2)\vdash 2$ and $m=1$,   the only  transposition is $(12)=(21)$. Therefore 
$$h_{0, {\;\ytableausetup
{boxsize=1.55em}
\ytableausetup
{aligntableaux=top}
\Yboxdim{5pt}
\tiny\yng(2)}}=\frac{1}{2}\cdot 1=\frac{1}{2}.$$
This example also shows that Hurwitz numbers  can be rational and not always a positive integer.
\item If $\mu=(3)\vdash 3$,  $m=2$  there  exist  $3$ transposition factorizations of the  three-cycle  $(123)=(12)(13):=(23)(21):=(31)(32)$   and we have  $3\cdot 2$ three-cycles in $S_3$ corresponding to connected trees. Thus 
$$h_{0, {\;\ytableausetup
{boxsize=1.55em}
\ytableausetup
{aligntableaux=top}
\Yboxdim{5pt}
\tiny\yng(3)}}=\frac{1}{6}(3\cdot 2)=1.$$

\item If $\mu=(2, 1)\vdash 3$ and $m=3$ we have $3^3$ triples of transpositions but $3$ of the triples consists of coinciding transpositions  and thus the corresponding  covering surface is not  connected. This implies that the single Hurwitz number
$$h_{0, {\;\ytableausetup
{boxsize=1.55em}
\ytableausetup
{aligntableaux=top}
\Yboxdim{5pt}
\tiny\yng(2,1)}}=\frac{1}{6}(3^3-3)=4.$$
\end{enumerate}
\end{example}

Now, since  for  $\mu=(d)$ the graph  $G_{\mu}$ is a tree,  assuming bijective results \cite{PM89} the corresponding Hurwitz number follows  immediately  from {\bf Cayley's formula} of 1860 for enumeration of trees. (Observe,  the Cayley formula  in the language  of transpositions, is  attributed to the Hungarian mathematician  D\'enes \cite{JD59}).
\begin{theorem}[D\'enes]
\label{JD}
There exist $d^{d-2}$ transposition factorization  of an $d$-cycle into
$d-1$ distinct  transpositions.

\end{theorem}
In the case $m=2$,  V.I. Arnol'd \cite{VA96} found the corresponding Hurwitz number  by using the notion of complex trigonometric polynomials.
\begin{theorem}[Arnol'd ]
For a partition  $\mu=(\mu_1, \mu_2)\vdash d$  the number of distinct minimal transitive transposition factorizations of $\sigma$ whose cycle type  equals $\mu$  is
\begin{equation}
\label{AV}
\mu_1^{\mu_1}\,\mu_2^{\mu_2}\,\frac{(\mu_1+\mu_2-1)!}{(\mu_1-1)!\,(\mu_2-1)!}\end{equation}
\end{theorem}
Still another case was settled  not that long ago by two physicists M. Crescimanno and W. Taylor.
\begin{theorem}[Crescimano-Taylor]
If  $m=d$ means $\mu=(1^d)$ i.e. the factorization of the identity, then the  number of distinct minimal transitive factorizations into transpositions
\begin{equation}
(2d -2)! \; d^{d-3}
\label{CT}
\end{equation}
\end{theorem}
was discovered  in \cite{CT95}, who asked apparently  asked the combinatorialist Richard Stanley who consulted  Goulden-Jackson about the result. Finally, Goulden-Jackson also independently \cite{GJ97, GJ99-1} discovered and  proved  the Hurwitz formula in its complete generality.

\section{Hurwitz numbers and the  symmetric groups }
Let $\CC[S_d]$ be the  group algebra of $S_d$. The group algebra  $\CC[S_d]$ has  is  $d!$ dimensional over $\CC$.  For each partition $\mu$ of $d$, denote by $C_{\mu}\in \CC[S_d]$ the basis elements in $S_d$, i.e. the sum of all permutations in $S_d$ of cyclic type $\mu$. 
We will denote by $C_e$ the class $C_{1^d}=C_{(1, 1,\ldots, 1)}$ of the identity permutation, which is the unit of the algebra $\CC[S_d]$, and $C_{2}$ for the sum of $C_{(2, 1,\ldots, 1)}$ of all transpositions. 

\begin{proposition}
Let $S_d$ denote the symmetric group of permutations of $d$ elements. For each partition $\mu$ of $d$, the sum of all elements in $S_d$ of cyclic type $\mu$ span the  center of $\CC[S_d]$.
\end{proposition}

\begin{example}
The center of the group algebra $\CC[S_d]$ is spanned by the three elements
\begin{align*}
 C_e =&\mathds{1}, \\
 C_{2} =& (12) + (23) + (13), \\
 C_{(3)} = &(123) + (132).
 \end{align*}
\end{example}

The  disconnected simple Hurwitz numbers possess the following natural interpretation. 

\begin{theorem}
The product of the class $C_\mu$ with the $wth$ power of the class $C_2$. Then
\begin{equation}
h_{g, \mu}^{\bullet}=\frac{1}{d!} \left[ C_e\right] C_\mu  C_2^w
\end{equation}
where $\left[ C_e\right] C_\mu  C_2^w$ is the coefficient of $C_e$ in the product $C_\mu\circ C_2^w$.
\end{theorem}

\begin{example}
Given $d=3, g=1$ and $\mu=(3)$ and $w=2g-2+\ell(\mu)+d=2\cdot 1-2+1+3=4$ we have
\begin{align*}
 C_{(3)}  C_2^4=& \left( (123) + (132) \right)\left( (12) + (23) + (13)\right) \\
 =& 54e +27 \left( (123) + (132)\right)\\
 =&54C_e+27  C_{(3)}.
\end{align*}

Thus 

$$h^\bullet_{g, _{\;\ytableausetup
{boxsize=1.55em}
\ytableausetup
{aligntableaux=top}
\Yboxdim{5pt}
\tiny\yng(3)}}=h_{g, _{\;\ytableausetup
{boxsize=1.55em}
\ytableausetup
{aligntableaux=top}
\Yboxdim{5pt}
\tiny\yng(3)}}=\frac{1}{3!} \left[ C_e\right] C_{(3)}  C_2^4=\frac{54}{6}=9.$$

\end{example}
Note that in this case, the connected and disconnected cases the calculations  are identical as  the special branch point has exactly one part i.e. $(3)\vdash 3$.

\subsection{ Burnside character formula}
Calculating Hurwitz numbers is multiplication problem in conjugacy class basis on the center of the group algebra $\CC[S_d]$.
Recall that both the conjugacy classes and irreducible representations of $S_{d}$ are in one to one correspondance with partitions of $d$. Using   Burnside formula in eqref{burnform}, we obtain a closed formula by involving a basis form  irreducible representation of   $S_d$

 \begin{theorem}[Burnside character formula]
 Let $\rho$ be an irreducible representation of $S_d$. 
 Denote the  character  of  $\rho$  by $\chi^\lambda$.  
The single Hurwitz numbers
\begin{equation}
\text{}\; h^\bullet_{g, \mu} =\sum_{\substack{\lambda \vdash d\\ }} \left( \frac{\dim \lambda}{|\lambda|!} \right)^2 f^\lambda(C_\mu) f^\lambda(C_{2})^w
\end{equation}
where, 
\begin{align*}
f^\lambda(C_{\mu^i})=&\frac{|C_{\mu^i}|}{\dim \lambda} \chi^\lambda(\mu^i). \\[2mm]
\end{align*}

 \end{theorem}
This connection was already known to A. Hurwitz in 1902 and has provided a rich interplay between geometry and combinatorics for a long time.

\subsection{  Hurwitz Monodromy Group}
Recall that  the single Hurwitz  number  corresponding to a fixed branching data is given by the degree of the covering map
\begin{equation}
\Phi: \Hh_{g, \mu}\longrightarrow  \Sym^w \PP^1\backslash \Delta
\label{brmap}
\end{equation}
It is an unsolved  problem to determine  the image, that is the monodromy group for branching morphism as described in \eqref{HF1} called the {\bf Hurwitz monodromy group}.  However,  in special cases  see \cite{EEHS91}  a good description can be obtained. This cases include the Hurwitz spaces $\HH_{g,\mu}$.  The  image of the fundamental group $\pi_1(\Sym^w \PP^1\setminus \Delta_w)$ to the symmetric group $S_{h_{g, \mu}}$ (where ${h_{g, \mu}}$ is single hurwitz number) is the Hurwitz monodromy group. Directly from the Hurwitz formulae in \eqref{Hurfomd} we have an intuitive indication, that  the Hurwitz monodromy groups are less than the full symmetric group at least for the first  nontrivial  cases $d=3$ and $4$, but nothing much we can say for $d>4$ from the shape of the formulae  seen earlier. Thus for $d=3$ or $4$, the Hurwitz monodromy groups  can be anticipated to have a structure which  heavily  reflects  the geometrical structure of $\mathbb F_2$ and $\mathbb F_2$ vector spaces.  \\[2mm]

Indeed,  as given in \eqref{Hurfomd} the single Hurwitz numbers $h_{g, _{\;\ytableausetup
{boxsize=1.55em}
\ytableausetup
{aligntableaux=top}
\Yboxdim{5pt}
\tiny\yng(2,1)}}$, 
$h_{g, _{\;\ytableausetup
{boxsize=1.55em}
\ytableausetup
{aligntableaux=top}
\Yboxdim{5pt}
\tiny\yng(2,1,1)}}$, 
 for degree $3$ and $4$ consists of the factors  $\frac{3^n-1}{2}$ and $2^{n}-1$. Recall that $\frac{3^n-1}{2}$  is the number of points in the $n-1$ dimensional projective space  over a field with $3$ elements and
$2^{n}-1$ is the number of points in a $n$ dimensional  projective space  over a field with $2$ elements.

 In fact, one way to compute  the single Hurwitz numbers $h_{g, _{\;\ytableausetup
{boxsize=1.55em}
\ytableausetup
{aligntableaux=top}
\Yboxdim{5pt}
\tiny\yng(2,1)}}$ and 
$h_{g, _{\;\ytableausetup
{boxsize=1.55em}
\ytableausetup
{aligntableaux=top}
\Yboxdim{5pt}
\tiny\yng(2,1,1)}}$
via a bijection between transpositions and elements of finite fields $\FF_2$ and $\FF_3$ respectively.

\begin{example}
To compute the $h_{g, _{\;\ytableausetup
{boxsize=1.55em}
\ytableausetup
{aligntableaux=top}
\Yboxdim{5pt}
\tiny\yng(2,1)}}$ for degree $3$ Hurwitz coverings, we establish a bijection between transpositions $t_1\ldots, t_w$ in $S_3$ specifying a  covering  curve $X$ with $w=2g+4$ branch points and the projective space of dimension $w-3$ over $\mathbb F_3$. This is easily obtained. Up to  conjugation  we can assume $t_1=(1,2)$ and consider the assignment
\begin{equation}
\mu : (12)\mapsto 0\quad (13)\mapsto 1, (23)\mapsto 2.
\end{equation}
Let $f^\ast((12))t_2\ldots t_{w})=(\mu (t_2), \ldots,\mu (t_{w-1}))$ we define the map $f$ from the projective points via
$$f(X)=f^\ast((t_2), \ldots,\mu (t_{w-1}))$$

As an example, if $(12)(13)(12)(23)(13))$ represent $X$, $f(X)=(1,0,2,1)$. One then can easily show the map $f$ is well defined from the requirement that the product of the transpositions must be identity, moreover its a bijection. 
Thus, 

$$h_{g, _{\;\ytableausetup
{boxsize=1.55em}
\ytableausetup
{aligntableaux=top}
\Yboxdim{5pt}
\tiny\yng(2,1)}}=\frac{1}{2}(3^{w-2}-1).$$
which is the number of points in the projective space of dimension $w-1$ over a field $\mathbb F_3$ with three elements.
\end{example}

\section{Hurwitz numbers in terms monodromy graphs}
We can now compute single Hurwitz numbers in terms of monodromy graphs. 
This is motivated from the definition of single Hurwitz numbers as equivalent to counting  permutation factorizations  into transpositions, we have another algebraic definition of single Hurwitz numbers   via enumeration of graphs. The presentation, we  follow the presentation in \cite{CJM} \\[2mm]

The  core behind  the derivation of this special case  is the fact that  multiplication of permutation by a transposition $\tau=(ab)$  can be easily understood; it  either  {\bf cuts} or {\bf joins} cycles of the permutation.  Namely, if $\sigma\in S_d$ has $m$ cycles then the product   $\tau\circ \sigma$ has either
\begin{enumerate}
\item {\bf Cut}:  $m-1$ cycles if $a$ and $b$ are in different cycles of $\sigma.$
\item {\bf Join}:  $m+1$ cycles if $a$ and $b$ are in same cycle of $\sigma.$
\end{enumerate}

\begin{figure}[htbp]
\begin{tikzpicture}[rotate=180, scale=0.8]
\tikzstyle{every node}=[font=\footnotesize]
\coordinate (s0) at (-5.0,  4) ;
\coordinate (s1) at (-3,  4);
\coordinate (s2) at (0, 4);
\coordinate (s3) at (3,4);
\coordinate (s4) at (6,  4);
\coordinate (s5) at (7.5,  4);
\coordinate (t0) at (-5.0,  3) ;
\coordinate (t1) at (-3,  3);
\coordinate (t2) at (-1.5,  3);
\coordinate (t3) at (4.5,  3);
\coordinate (t4) at (6,  3);
\coordinate (t5) at (7.5,  3);
\coordinate (b0) at (-5.0,  2) ;
\coordinate (b1) at (-3,  2);
\coordinate (b2) at (0,  2);
\coordinate (b3) at (3,  2);
\coordinate (b4) at (4.5,  2);
\coordinate (b5) at (7.5,  2);
\coordinate (ns1) at (-4.5,  4) ;
\coordinate (ns2) at (-1.5,  4);
\coordinate (ns3) at (1.5, 4);
\coordinate (ns4) at (4.5,4);
\coordinate (ns5) at (7.0,  4);
\coordinate (nt1) at (-4.5,  3) ;
\coordinate (nt2) at (-1.5,  3);
\coordinate (nt3) at (1.5,  3);
\coordinate (nt4) at (4.5,  3);
\coordinate (nt5) at (7.0,  3);
\coordinate (nb1) at (-4.5,  2) ;
\coordinate (nb2) at (-1.5,  2);
\coordinate (nb3) at (1.5,  2);
\coordinate (nb4) at (4.5,  2);
\coordinate (nb5) at (7.0,  2);
\begin{knot}
\strand[blue,ultra thick] (t1)
to [out=120, in=0] (s0);
\strand[blue,ultra thick] (b0)
to [out=0, in=250] (t1);
\draw [-, ultra thick, blue] (t1) -- (t2);
\end{knot}
\node[below] at (ns1){$m_2$};
\node[below] at (nt2){$m_1+m_2$};
\node[above=2pt] at (nb1){$m_1$};
\end{tikzpicture}
\hspace{2cm}
\begin{tikzpicture}[rotate=180, scale=0.8]
\tikzstyle{every node}=[font=\footnotesize]
\coordinate (s0) at (-5.0,  4) ;
\coordinate (s1) at (-3,  4);
\coordinate (s2) at (0, 4);
\coordinate (s3) at (3,4);
\coordinate (s4) at (6,  4);
\coordinate (s5) at (7.5,  4);
\coordinate (t0) at (-5.0,  3) ;
\coordinate (t1) at (-3,  3);
\coordinate (t2) at (-1.5,  3);
\coordinate (t3) at (4.5,  3);
\coordinate (t4) at (6,  3);
\coordinate (t5) at (7.5,  3);
\coordinate (b0) at (-5.0,  2) ;
\coordinate (b1) at (-3,  2);
\coordinate (b2) at (0,  2);
\coordinate (b3) at (3,  2);
\coordinate (b4) at (4.5,  2);
\coordinate (b5) at (7.5,  2);
\coordinate (ns1) at (-4.5,  4) ;
\coordinate (ns2) at (-1.5,  4);
\coordinate (ns3) at (1.5, 4);
\coordinate (ns4) at (4.5,4);
\coordinate (ns5) at (7.0,  4);
\coordinate (nt1) at (-4.5,  3) ;
\coordinate (nt2) at (-1.5,  3);
\coordinate (nt3) at (1.5,  3);
\coordinate (nt4) at (4.5,  3);
\coordinate (nt5) at (7.0,  3);
\coordinate (nb1) at (-4.5,  2) ;
\coordinate (nb2) at (-1.5,  2);
\coordinate (nb3) at (1.5,  2);
\coordinate (nb4) at (4.5,  2);
\coordinate (nb5) at (7.0,  2);
\begin{knot}
\strand[blue,ultra thick] (t1)
to [out=120, in=0] (s0);
\strand[blue,ultra thick] (b0)
to [out=0, in=250] (t1);
\draw [-, ultra thick, blue] (t1) -- (t2);
\end{knot}
\node[below] at (ns1){$m$};
\node[below] at (nt2){$2m$};
\node[above=2pt] at (nb1){$m$};
\end{tikzpicture}
\hspace{2cm}
\begin{tikzpicture}[yshift=-2cm, scale=0.8]
\tikzstyle{every node}=[font=\footnotesize]
\coordinate (s0) at (-5.0,  4) ;
\coordinate (s1) at (-3,  4);
\coordinate (s2) at (0, 4);
\coordinate (s3) at (3,4);
\coordinate (s4) at (6,  4);
\coordinate (s5) at (7.5,  4);
\coordinate (t0) at (-5.0,  3) ;
\coordinate (t1) at (-3,  3);
\coordinate (t2) at (-1.5,  3);
\coordinate (t3) at (4.5,  3);
\coordinate (t4) at (6,  3);
\coordinate (t5) at (7.5,  3);
\coordinate (b0) at (-5.0,  2) ;
\coordinate (b1) at (-3,  2);
\coordinate (b2) at (0,  2);
\coordinate (b3) at (3,  2);
\coordinate (b4) at (4.5,  2);
\coordinate (b5) at (7.5,  2);
\coordinate (ns1) at (-4.5,  4) ;
\coordinate (ns2) at (-1.5,  4);
\coordinate (ns3) at (1.5, 4);
\coordinate (ns4) at (4.5,4);
\coordinate (ns5) at (7.0,  4);
\coordinate (nt1) at (-4.5,  3) ;
\coordinate (nt2) at (-1.5,  3);
\coordinate (nt3) at (1.5,  3);
\coordinate (nt4) at (4.5,  3);
\coordinate (nt5) at (7.0,  3);
\coordinate (nb1) at (-4.5,  2) ;
\coordinate (nb2) at (-1.5,  2);
\coordinate (nb3) at (1.5,  2);
\coordinate (nb4) at (4.5,  2);
\coordinate (nb5) at (7.0,  2);
\begin{knot}
\strand[blue,ultra thick] (t1)
to [out=120, in=0] (s0);
\strand[blue,ultra thick] (b0)
to [out=0, in=250] (t1);
\draw [-, ultra thick, blue] (t1) -- (t2);
\end{knot}
\node[above] at (ns1){$m_1$};
\node[below] at (nt2){$m_1+m_2$};
\node[below=-1pt] at (nb1){$m_2$};
\end{tikzpicture}
\caption{Effects on  multiplicities  of a cycle type of $\sigma\in S_d$ in  the composition $\tau_m\circ \sigma$.}
\end{figure}
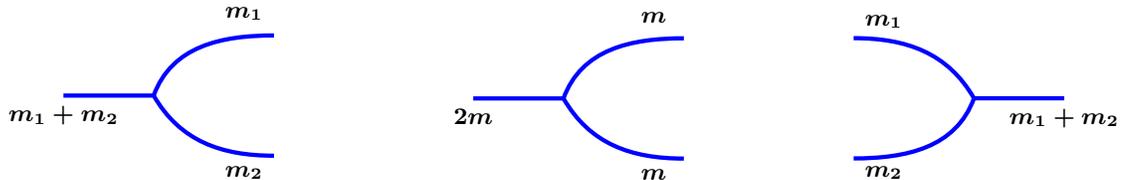

\begin{example}
The multiplication of  permutation $(12345)\in S_5$ on the left by $(14)$ gives $(15)(234)$. In other words, cuts  it  into two cycles. On the other hand,  multiplication of  the permutation  $(15)(234)$ on the left by $(14)$ joins the two cycles together.
\end{example}

 We associate a graph  called a  monodromy graphs which  project to the segment  $[\infty, 1\dots,w]$   with labeled  edges  as follows:\\[2mm]

\begin{definition}
Constructing the  monodromy graph $\Gamma$, project to the  segment  $[0, w]$
marked and labelled $\infty, 1\dots,w$ as follows:
\begin{enumerate}
\item[1)] Start with $n$ small strands over $\infty$ decorated by weights labels $\mu_1, \ldots, \mu_n$. 
\item[2)] \label{b} Over the point $1$ create a three-valent vertex by either joining two strands or splitting one with weight strictly greater than $1$.
\begin{itemize}
\item  [\textbf{Join:}] If a join, label the new strand with the sum of the weights of the edges joined.
\item [\textbf{Cut:}] If  a cut, label the two new strands in all possible (positive) ways adding to the weight of the split edge,
\end{itemize}

\item [3)]\label{c} In each case, consider a unique representative for any isomorphism class of labeled graphs.
\item [4)] Repeat (\ref{b}) and (\ref{c}) for all successive integers up to $w$,
\item [5)]Retain all connected graphs that “terminate” with weight  $(2)$ or $ (2, 1, \ldots, 1)\vdash d$ over $w$.
\end{enumerate}
\end{definition}

We obtain a  connected graph $\Gamma$ of genus  $g$  with a map to $[0, w]$. We call $\Gamma$ together with the map the \textit{monodromy graph of type} $(g,\mu)$ {\em corresponding  to} $(\sigma,\tau_1,\dots,\tau_w)$.

\begin{definition}
Given a monodromy graph, a {\bf balanced fork} is a tripod with weights $n, n, 2n$ such that the vertices of weight $n$ lie over $\infty$ or $w$.  A  {\bf wiener} consists of a strand of weight $2n$ splitting into two strands of weight $n$ and then re-joining
\end{definition}

\begin{figure}[htbp]
\centering
\begin{tikzpicture}[scale=1.0]
\tikzstyle{every node}=[font=\footnotesize]
\coordinate (s0) at (-5.0,  4) ;
\coordinate (s1) at (-3,  4);
\coordinate (s2) at (0, 4);
\coordinate (s3) at (3,4);
\coordinate (s4) at (6,  4);
\coordinate (s5) at (7.5,  4);
\coordinate (t0) at (-5.0,  3) ;
\coordinate (t1) at (-3,  3);
\coordinate (t2) at (0,  3);
\coordinate (t3) at (3,  3);
\coordinate (t15) at (-2,  3);
\coordinate (t17) at (-1,  3);
\coordinate (t35) at (4,  3);
\coordinate (t37) at (5,  3);
\coordinate (t4) at (6,  3);
\coordinate (t5) at (7.5,  3);
\coordinate (b0) at (-5.0,  2) ;
\coordinate (b1) at (-3,  2);
\coordinate (b2) at (0,  2);
\coordinate (b3) at (3,  2);
\coordinate (b4) at (4.5,  2);
\coordinate (b5) at (7.5,  2);
\coordinate (ns1) at (-4.5,  4) ;
\coordinate (ns2) at (-1.5,  4);
\coordinate (ns3) at (1.5, 4);
\coordinate (ns4) at (4.5,4);
\coordinate (ns5) at (7.0,  4);
\coordinate (nt1) at (-4.5,  3) ;
\coordinate (nt2) at (-1.5,  3);
\coordinate (nt3) at (1.5,  3);
\coordinate (nt4) at (4.5,  3);
\coordinate (nt5) at (7.0,  3);
\coordinate (nb1) at (-4.5,  2) ;
\coordinate (nb2) at (-1.5,  2);
\coordinate (nb3) at (1.5,  2);
\coordinate (nb4) at (4.5,  2);
\coordinate (nb5) at (7.0,  2);
\filldraw (t1)circle (3pt);
\filldraw (t2)circle (3pt);
\filldraw (t3)circle (3pt);
\filldraw (t4)circle (3pt);

\begin{knot}

\strand[blue,ultra thick] (t1)
to [out=120, in=0] (s0);

\draw [-, ultra thick, blue] (t1) -- (t15);
\draw [-, ultra thick, blue] (t17) -- (t2);

\strand[blue,ultra thick] (t2)
to [out=-70,in=-120] (t3)
to [out=120,in=60] (t2);

\draw [-, ultra thick, blue] (t3) -- (t35);
\draw [-, ultra thick, blue] (t37) -- (t4);

\strand[blue,ultra thick] (b0)
to [out=0, in=250] (t1);

\strand[blue,ultra thick] (t4)
to [out=60, in=180] (s5);
\strand[blue,ultra thick] (t4)
to [out=-60, in=180] (b5);

\end{knot}
\node (A) at (1, 1.8) {};
\node (B) at (1, -0.8) {};

\coordinate (p15) at (-2,  -0);
\coordinate (p17) at (-1,  -0);
\coordinate (p35) at (4,  -0);
\coordinate (p37) at (5,  -0);

\draw [-, ultra thick] (-5.5, -0) -- (p15);
\draw [-, ultra thick] (p17) -- (p35);
\draw [-, ultra thick] (p37) -- (7.8,-0);

\filldraw[red] (-5.3,-0) circle (.12) node[below=3pt] {$\infty$};
\filldraw[red] (-3,-0) circle (.12) node[below=3pt]{};
\filldraw[red] (0,-0) circle (.12) node[below=3pt] {};
\filldraw[red](3,-0) circle (.12) node[below=3pt] {};
\filldraw[red] (6,-0)circle (.12) node[below=3pt] {};
\filldraw[red] (7.5,-0)circle (.12) node[below=3pt] {$w$};
\node[below] at (ns1){$n$};
\node[above=2pt] at (nb1){$n$};
\node[below] at (-2.2, 3){$2n$};

\node[below] at (-0.8, 3){$2n$};
\node[below=5pt] at (ns3){$n$};
\node[above=6pt] at (nb3){$n$};
\node[below] at (3.8, 3){$2n$};

\node[below] at (ns5){$n$};
\node[above=2pt] at (nb5){$n$};
\node[below] at (5.2, 3){$2n$};
  \end{tikzpicture} 
\caption{Local structure for balanced left pointing forks,  wiener and balanced right pointing forks }
\end{figure}
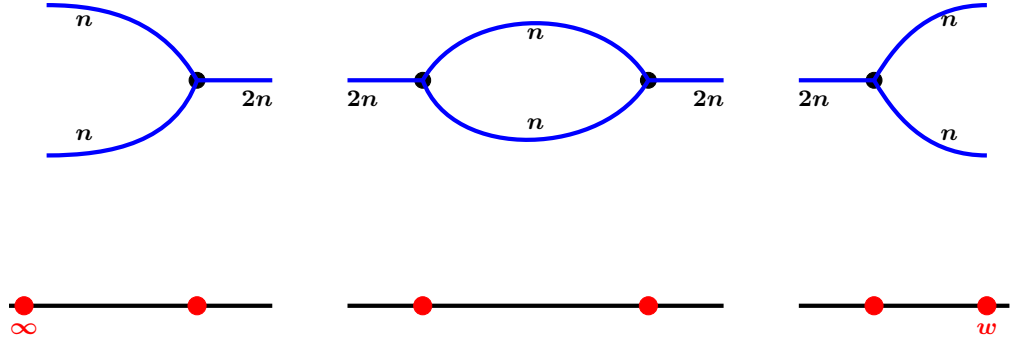
The map $\Gamma\to [0, w]$ can be viewed as a tropical cover of degree $d$, where the edges adjacent to vertices over $\infty$ yield the profile $\mu$.
The {\em balancing condition for monodromy graphs}, comes from the observation that by definition a monodromy graph is a combinatorial type of a tropical morphism see \cite{CJM} for more details.

\begin{definition}
 An isomorphism of monodromy graphs $\Gamma_1\to[0, w]$ and $\Gamma_2\to[0, w]$ of type $(g,\mu)$ is a graph isomorphism $f:\Gamma_1\to\Gamma_2$, such that
\begin{equation}
\begin{tikzcd}[column sep=0.7cm,row sep=0.9cm]
\Gamma_1 \arrow{rr}{f} \arrow {dr} & & \arrow{dl} \Gamma_2\\
& {[0, w]} &
\end{tikzcd}
\end{equation}
commutes.
\end{definition}

\begin{table}[htbp]
\caption{Illustration of Hurwitz numbers using Monodromy graphs for degree $3$.}
\label{mon-cubic}
\setlength{\tabcolsep}{5mm} 
\def\arraystretch{1.25} 
\resizebox{\textwidth}{!}{\begin{tabular}{|c|c|c|}
  \hline
 Graph type    &  $\frac{\prod \omega(e)}{\text{Aut}(\Gamma)}$    &  Contribution
  \\ \hline
 \begin{tikzpicture}[baseline=0, scale=0.7]
\tikzstyle{every node}=[font=\footnotesize]
\coordinate (s0) at (-5.0,  4) ;
\coordinate (s1) at (-3,  4);
\coordinate (s2) at (0, 4);
\coordinate (s3) at (3,4);
\coordinate (s4) at (6,  4);
\coordinate (s5) at (7.5,  4);
\coordinate (t0) at (-5.0,  3) ;
\coordinate (t1) at (-3,  3);
\coordinate (t2) at (0,  3);
\coordinate (t3) at (3,  3);
\coordinate (t4) at (6,  3);
\coordinate (t5) at (7.5,  3);
\coordinate (b0) at (-5.0,  2) ;
\coordinate (b1) at (-3,  2);
\coordinate (b2) at (0,  2);
\coordinate (b3) at (3,  2);
\coordinate (b4) at (4.5,  2);
\coordinate (b5) at (7.5,  2);
\coordinate (ns1) at (-4.5,  4) ;
\coordinate (ns2) at (-1.5,  4);
\coordinate (ns3) at (1.5, 4);
\coordinate (ns4) at (4.5,4);
\coordinate (ns5) at (7.0,  4);
\coordinate (nt1) at (-4.5,  3) ;
\coordinate (nt2) at (-1.5,  3);
\coordinate (nt3) at (1.5,  3);
\coordinate (nt4) at (4.5,  3);
\coordinate (nt5) at (7.0,  3);
\coordinate (nb1) at (-4.5,  2) ;
\coordinate (nb2) at (-1.5,  2);
\coordinate (nb3) at (1.5,  2);
\coordinate (nb4) at (4.5,  2);
\coordinate (nb5) at (7.0,  2);
\filldraw (t1)circle (3pt);
\filldraw (t2)circle (3pt);
\filldraw (t3)circle (3pt);
\filldraw (t4)circle (3pt);
\begin{knot}
\strand[blue,ultra thick] (t1)
to [out=120, in=0] (s0);

\strand[blue,ultra thick] (t1)
to [out=-70,in=-120] (t2)
to [out=120,in=60] (t1);

\draw [-, ultra thick, blue] (t2) -- (t4);

\strand[blue,ultra thick] (b0)
to [out=0, in=180] (b2)
to [out=0, in=250] (t3);

\strand[blue,ultra thick] (t4)
to [out=60, in=180] (s5);
\strand[blue,ultra thick] (t4)
to [out=-60, in=180] (b5);

\draw [-, ultra thick] (-5.5, -1) -- (7.5,-1);
\end{knot}
\node (A) at (1, 1.8) {};
\node (B) at (1, -0.8) {};
\node (L) at (8.5, -1) {$\PP^1$};
\filldraw[red] (-5.3,-1) circle (.12) node[below=3pt] {$\infty$};
\filldraw[red] (-3,-1) circle (.12) node[below=3pt]{$1$};
\filldraw[red] (0,-1) circle (.12) node[below=3pt] {};
\filldraw[red](3,-1) circle (.12) node[below=3pt] {};
\filldraw[red] (6,-1)circle (.12) node[below=3pt] {};
\filldraw[red] (7.5,-1)circle (.12) node[below=3pt] {$w$};
\node[below] at (ns1){2};
\node[below=-6pt] at (ns2){1};
\node[below=2pt] at (ns5){2};
\node[below=6pt] at (nt2){1};
\node[below] at (nt3){2};
\node[below] at (nt4){3};
\node[above=2pt] at (nb1){1};
\node[above=2pt] at (nb5){1};
\draw[->, ultra thick] node[above]{} (A) -- node[right] {$f$} (B);
\end{tikzpicture}
 
  &   $\begin{aligned}
  =&\frac{1\cdot1\cdot 2\cdot 3}{2}\\
  =&3\\
  \end{aligned}
  $   & $3\cdot 2=6$
  \\ \hline
  \begin{tikzpicture}[scale=0.7]
\tikzstyle{every node}=[font=\footnotesize]
\coordinate (s0) at (-5.0,  4) ;
\coordinate (s1) at (-3,  4);
\coordinate (s2) at (0, 4);
\coordinate (s3) at (3,4);
\coordinate (s4) at (6,  4);
\coordinate (s5) at (7.5,  4);
\coordinate (t0) at (-5.0,  3) ;
\coordinate (t1) at (-3,  3);
\coordinate (t2) at (0,  3);
\coordinate (t3) at (3,  3);
\coordinate (t4) at (6,  3);
\coordinate (t5) at (7.5,  3);
\coordinate (b0) at (-5.0,  2) ;
\coordinate (b1) at (-3,  2);
\coordinate (b2) at (0,  2);
\coordinate (b3) at (3,  2);
\coordinate (b4) at (4.5,  2);
\coordinate (b5) at (7.5,  2);
\coordinate (ns1) at (-4.5,  4) ;
\coordinate (ns2) at (-1.5,  4);
\coordinate (ns3) at (1.5, 4);
\coordinate (ns4) at (4.5,4);
\coordinate (ns5) at (7.0,  4);
\coordinate (nt1) at (-4.5,  3) ;
\coordinate (nt2) at (-1.5,  3);
\coordinate (nt3) at (1.5,  3);
\coordinate (nt4) at (4.5,  3);
\coordinate (nt5) at (7.0,  3);

\coordinate (nb1) at (-4.5,  2) ;
\coordinate (nb2) at (-1.5,  2);
\coordinate (nb3) at (1.5,  2);
\coordinate (nb4) at (4.5,  2);
\coordinate (nb5) at (7.0,  2);

\filldraw (t1)circle (3pt);
\filldraw (t2)circle (3pt);
\filldraw (t3)circle (3pt);
\filldraw (t4)circle (3pt);

\begin{knot}

\strand[blue,ultra thick] (t1)
to [out=120, in=0] (s0);

\strand[blue,ultra thick] (t1)
to [out=-70,in=-120] (t2)
to [out=120,in=60] (t1);

\draw [-, ultra thick, blue] (t2) -- (t4);
\strand[blue,ultra thick] (b0)
to [out=0, in=180] (b3)
to [out=0, in=250] (t4);
\strand[blue,ultra thick] (t3)
to [out=60, in=180] (s5);
\strand[blue,ultra thick] (t4)
to [out=0, in=180] (t5);
\draw [-, ultra thick] (-5.5, -1) -- (7.5,-1);
\end{knot}
\node (A) at (1, 1.8) {};
\node (B) at (1, -0.8) {};
\node (L) at (8.5, -1) {$\PP^1$};
\filldraw[red] (-5.3,-1) circle (.12) node[below=3pt] {$\infty$};
\filldraw[red] (-3,-1) circle (.12) node[below=3pt]{$1$};
\filldraw[red] (0,-1) circle (.12) node[below=3pt] {};
\filldraw[red](3,-1) circle (.12) node[below=3pt] {};
\filldraw[red] (6,-1)circle (.12) node[below=3pt] {};
\filldraw[red] (7.5,-1)circle (.12) node[below=3pt] {$w$};
\node[below] at (ns1){2};
\node[below=-6pt] at (ns2){1};
\node[below=2pt] at (ns5){1};
\node[below=6pt] at (nt2){1};
\node[below] at (nt3){2};
\node[below] at (nt4){1};
\node[below] at (nt5){2};
\node[above=2pt] at (nb1){1};
\draw[->, ultra thick] node[above]{} (A) -- node[right] {$f$} (B);
  \end{tikzpicture}  
  
  &   $\begin{aligned}
  =&\frac{1\cdot1\cdot 2\cdot 1}{2}\\
  =&1\\
  \end{aligned}
  $   & $1\cdot 2=2$
  \\ \hline
\begin{tikzpicture}[scale=0.7]
\tikzstyle{every node}=[font=\footnotesize]
\coordinate (s0) at (-5.0,  4) ;
\coordinate (s1) at (-3,  4);
\coordinate (s2) at (0, 4);
\coordinate (s3) at (3,4);
\coordinate (s4) at (6,  4);
\coordinate (s5) at (7.5,  4);
\coordinate (t0) at (-5.0,  3) ;
\coordinate (t1) at (-3,  3);
\coordinate (t2) at (0,  3);
\coordinate (t3) at (3,  3);
\coordinate (t4) at (6,  3);
\coordinate (t5) at (7.5,  3);
\coordinate (b0) at (-5.0,  2) ;
\coordinate (b1) at (-3,  2);
\coordinate (b2) at (0,  2);
\coordinate (b3) at (3,  2);
\coordinate (b4) at (4.5,  2);
\coordinate (b5) at (7.5,  2);
\coordinate (ns1) at (-4.5,  4) ;
\coordinate (ns2) at (-1.5,  4);
\coordinate (ns3) at (1.5, 4);
\coordinate (ns4) at (4.5,4);
\coordinate (ns5) at (7.0,  4);
\coordinate (nt1) at (-4.5,  3) ;
\coordinate (nt2) at (-1.5,  3);
\coordinate (nt3) at (1.5,  3);
\coordinate (nt4) at (4.5,  3);
\coordinate (nt5) at (7.0,  3);

\coordinate (nb1) at (-4.5,  2) ;
\coordinate (nb2) at (-1.5,  2);
\coordinate (nb3) at (1.5,  2);
\coordinate (nb4) at (4.5,  2);
\coordinate (nb5) at (7.0,  2);

\filldraw (t1)circle (3pt);
\filldraw (t2)circle (3pt);
\filldraw (t3)circle (3pt);
\filldraw (t4)circle (3pt);

\begin{knot}

\strand[blue,ultra thick] (t1)
to [out=120, in=0] (s0);

\strand[blue,ultra thick] (t1)
to [out=60, in=180] (-1.5, 4)
to [out=0, in=180] (1.5, 4)
to [out=0,in=120] (t3);

\strand[blue,ultra thick] (t1)
to [out=0,in=180] (t2);

\draw [-, ultra thick, blue] (t2) -- (t4);

\strand[blue,ultra thick] (b0)
to [out=0, in=180] (b1)
to [out=0, in=250] (t2);

\strand[blue,ultra thick] (t4)
to [out=60, in=180] (s5);
\strand[blue,ultra thick] (t4)
to [out=-60, in=180] (b5);
\draw [-, ultra thick] (-5.5, -1) -- (7.5,-1);
\end{knot}
\node (A) at (1, 1.8) {};
\node (B) at (1, -0.8) {};
\node (L) at (8.5, -1) {$\PP^1$};
\filldraw[red] (-5.3,-1) circle (.12) node[below=3pt] {$\infty$};
\filldraw[red] (-3,-1) circle (.12) node[below=3pt]{$1$};
\filldraw[red] (0,-1) circle (.12) node[below=3pt] {};
\filldraw[red](3,-1) circle (.12) node[below=3pt] {};
\filldraw[red] (6,-1)circle (.12) node[below=3pt] {};
\filldraw[red] (7.5,-1)circle (.12) node[below=3pt] {$w$};
\node[below] at (ns1){2};
\node[below=-6pt] at (ns2){1};
\node[below=2pt] at (ns5){2};
\node[below=-2pt] at (nt2){1};
\node[below] at (nt3){2};
\node[below] at (nt4){3};

\node[above=2pt] at (nb1){1};
\node[above=2pt] at (nb5){1};
\draw[->, ultra thick] node[above]{} (A) -- node[right] {$f$} (B);
\end{tikzpicture}

  &   $\begin{aligned}
  =&\frac{1\cdot1\cdot 2\cdot 3}{1}\\
  =&6\\
  \end{aligned}
  $   & $6\cdot 2=12$
  \\ \hline
 \begin{tikzpicture}[scale=0.7]
 \tikzstyle{every node}=[font=\footnotesize]
\coordinate (s0) at (-5.0,  4) ;
\coordinate (s1) at (-3,  4);
\coordinate (s2) at (0, 4);
\coordinate (s3) at (3,4);
\coordinate (s4) at (6,  4);
\coordinate (s5) at (7.5,  4);
\coordinate (t0) at (-5.0,  3) ;
\coordinate (t1) at (-3,  3);
\coordinate (t2) at (0,  3);
\coordinate (t3) at (3,  3);
\coordinate (t4) at (6,  3);
\coordinate (t5) at (7.5,  3);
\coordinate (b0) at (-5.0,  2) ;
\coordinate (b1) at (-3,  2);
\coordinate (b2) at (0,  2);
\coordinate (b3) at (3,  2);
\coordinate (b4) at (4.5,  2);
\coordinate (b5) at (7.5,  2);
\coordinate (ns1) at (-4.5,  4) ;
\coordinate (ns2) at (-1.5,  4);
\coordinate (ns3) at (1.5, 4);
\coordinate (ns4) at (4.5,4);
\coordinate (ns5) at (7.0,  4);
\coordinate (nt1) at (-4.5,  3) ;
\coordinate (nt2) at (-1.5,  3);
\coordinate (nt3) at (1.5,  3);
\coordinate (nt4) at (4.5,  3);
\coordinate (nt5) at (7.0,  3);

\coordinate (nb1) at (-4.5,  2) ;
\coordinate (nb2) at (-1.5,  2);
\coordinate (nb3) at (1.5,  2);
\coordinate (nb4) at (4.5,  2);
\coordinate (nb5) at (7.0,  2);

\filldraw (t1)circle (3pt);
\filldraw (t2)circle (3pt);
\filldraw (t3)circle (3pt);
\filldraw (t4)circle (3pt);

\begin{knot}

\strand[blue,ultra thick] (t1)
to [out=120, in=0] (s0);

\strand[blue,ultra thick] (t1)
to [out=60, in=180] (s2)
to [out=0, in=180] (s3)
to [out=0,in=120] (t4);

\strand[blue,ultra thick] (t1)
to [out=0,in=180] (t2);

\draw [-, ultra thick, blue] (t2) -- (t4);

\strand[blue,ultra thick] (b0)
to [out=0, in=180] (b1)
to [out=0, in=250] (t2);

\strand[blue,ultra thick] (t4)
to [out=60, in=180] (s5);
\strand[blue,ultra thick] (t3)
to [out=-60, in=180] (b4)
to [out=0, in=180] (b5);

\draw [-, ultra thick] (-5.5, -1) -- (7.5,-1);
\end{knot}
\node (A) at (1, 1.8) {};
\node (B) at (1, -0.8) {};
\node (L) at (8.5, -1) {$\PP^1$};
\filldraw[red] (-5.3,-1) circle (.12) node[below=3pt] {$\infty$};
\filldraw[red] (-3,-1) circle (.12) node[below=3pt]{$1$};
\filldraw[red] (0,-1) circle (.12) node[below=3pt] {};
\filldraw[red](3,-1) circle (.12) node[below=3pt] {};
\filldraw[red] (6,-1)circle (.12) node[below=3pt] {};
\filldraw[red] (7.5,-1)circle (.12) node[below=3pt] {$w$};
\node[below] at (ns1){2};
\node[below=-6pt] at (ns2){1};
\node[below=2pt] at (ns5){2};
\node[below=-2pt] at (nt2){1};
\node[below] at (nt3){2};
\node[below] at (nt4){1};
\node[above=2pt] at (nb1){1};
\node[above=2pt] at (nb5){1};
\draw[->, ultra thick] node[above]{} (A) -- node[right] {$f$} (B);    
  \end{tikzpicture}   
   &   $\begin{aligned}
  =&\frac{1\cdot1\cdot 2\cdot 1}{1}\\
  =&2\\
  \end{aligned}
  $   & $2\cdot 1=2$
  \\ \hline
  
\begin{tikzpicture}[scale=0.7]
\tikzstyle{every node}=[font=\footnotesize]
\coordinate (s0) at (-5.0,  4) ;
\coordinate (s1) at (-3,  4);
\coordinate (s2) at (0, 4);
\coordinate (s3) at (3,4);
\coordinate (s4) at (6,  4);
\coordinate (s5) at (7.5,  4);
\coordinate (t0) at (-5.0,  3) ;
\coordinate (t1) at (-3,  3);
\coordinate (t2) at (0,  3);
\coordinate (t3) at (3,  3);
\coordinate (t4) at (6,  3);
\coordinate (t5) at (7.5,  3);
\coordinate (b0) at (-5.0,  2) ;
\coordinate (b1) at (-3,  2);
\coordinate (b2) at (0,  2);
\coordinate (b3) at (3,  2);
\coordinate (b4) at (4.5,  2);
\coordinate (b5) at (7.5,  2);
\coordinate (ns1) at (-4.5,  4) ;
\coordinate (ns2) at (-1.5,  4);
\coordinate (ns3) at (1.5, 4);
\coordinate (ns4) at (4.5,4);
\coordinate (ns5) at (7.0,  4);
\coordinate (nt1) at (-4.5,  3) ;
\coordinate (nt2) at (-1.5,  3);
\coordinate (nt3) at (1.5,  3);
\coordinate (nt4) at (4.5,  3);
\coordinate (nt5) at (7.0,  3);

\coordinate (nb1) at (-4.5,  2) ;
\coordinate (nb2) at (-1.5,  2);
\coordinate (nb3) at (1.5,  2);
\coordinate (nb4) at (4.5,  2);
\coordinate (nb5) at (7.0,  2);

\filldraw (t1)circle (3pt);
\filldraw (t2)circle (3pt);
\filldraw (t3)circle (3pt);
\filldraw (t4)circle (3pt);

\begin{knot}

\strand[blue,ultra thick] (t1)
to [out=120, in=0] (s0);

\draw [-, ultra thick, blue] (t1) -- (t2);
\strand[blue,ultra thick] (t2)
to [out=-70,in=-120] (t3)
to [out=120,in=60] (t2);

\draw [-, ultra thick, blue] (t3) -- (t4);

\strand[blue,ultra thick] (b0)
to [out=0, in=250] (t1);

\strand[blue,ultra thick] (t4)
to [out=60, in=180] (s5);
\strand[blue,ultra thick] (t4)
to [out=-60, in=180] (b5);

\draw [-, ultra thick] (-5.5, -1) -- (7.5,-1);
\end{knot}
\node (A) at (1, 1.8) {};
\node (B) at (1, -0.8) {};
\node (L) at (8.5, -1) {$\PP^1$};
\filldraw[red] (-5.3,-1) circle (.12) node[below=3pt] {$\infty$};
\filldraw[red] (-3,-1) circle (.12) node[below=3pt]{$1$};
\filldraw[red] (0,-1) circle (.12) node[below=3pt] {};
\filldraw[red](3,-1) circle (.12) node[below=3pt] {};
\filldraw[red] (6,-1)circle (.12) node[below=3pt] {};
\filldraw[red] (7.5,-1)circle (.12) node[below=3pt] {$w$};
\node[below] at (ns1){2};
\node[below=5pt] at (ns3){2};
\node[below] at (ns5){2};

\node[below] at (nt2){3};
\node[below] at (nt4){3};

\node[above=2pt] at (nb1){1};
\node[above=6pt] at (nb3){1};
\node[above=2pt] at (nb5){1};

\draw[->, ultra thick] node[above]{} (A) -- node[right] {$f$} (B);
  \end{tikzpicture}   
  &   $\begin{aligned}
  =&\frac{3\cdot 2\cdot 1\cdot 3}{1}\\
  =&18\\
  \end{aligned}
  $   & $18\cdot 1=18$
  \\ \hline

  \end{tabular}
  }
\end{table}

\begin{definition}
Let $\BB_{\mu}$ denote the set of all isomorphism classes of monodromy graphs.  Then the single Hurwitz number 
\label{TropHur}
\begin{equation}
h_{g,\mu}=\sum_{[\Gamma\in \BB_{\mu} ]}\; \frac{1}{|\Aut(\Gamma)|},
\end{equation}
is the number of monodromy graphs $\Gamma$ in $\BB_{\mu}$  divided by $|\Aut(\Gamma)|$.
\end{definition}

We simplify the definition \ref{TropHur} above  to give the Cavalieri-Johnson-Markwig  formula in \cite{CJM}.

\begin{proposition}
\label{TropHur}
The  Hurwitz number $h_{g, \mu}$is computed as a weighted sum over monodromy graphs.

 \begin{equation}
h_{g, \mu}=\sum_{[\Gamma\in \BB_{\mu} ]}\;\frac{1}{|\Aut(\Gamma)|}\prod \omega(e)
\end{equation}
where we take the product of all the interior edge weights $\omega(e)$. The the automorphism group $\Aut(\Gamma)$  involve  factors of $1/2$ coming from the balanced forks and wieners of $\Gamma$.
\end{proposition}

\begin{example}
If $g=1,d=3$ and $\mu=(2, 1)\vdash 3$, we will show  that  
$$
h_{1, _{\;\ytableausetup
{boxsize=1.55em}
\ytableausetup
{aligntableaux=top}
\Yboxdim{5pt}
\tiny\yng(2,1)}} =40.$$
That is there is a family of $40$ non-isomorphic  cubics  over $6$ points in $\PP^1$. Observe that in this case $g=1,d=3$, then by Riemann-Hurwitz formula 
give us  $$w=2g-2+\ell(\mu)+d=2\cdot 1-2+2+3=5.$$ 

Note that in the computation of the total contribution, graphs which have a vertical  symmetry will yield another representative. Thus  we  multiply the factor $\displaystyle \frac{\prod \omega(e)}{\text{Aut}(\Gamma)}$ by $2$ to factor for this contribution. Refer to Table \ref{mon-cubic} for the specific monodromy graphs.

\begin{align*}
h_{1, _{\;\ytableausetup
{boxsize=1.55em}
\ytableausetup
{aligntableaux=top}
\Yboxdim{5pt}
\tiny\yng(2,1)}}=&6 + 2 + 12 + 2 + 18\\ 
= &40.
\end{align*}

\end{example}

Single Hurwitz numbers turn out to be closely related to the intersection theory on the moduli space of stable curves. We formulate remarkable ELSV formula \cite{ELSV99, ELSV01} following  a result of  {\bf E}kedahl-{\bf L}ando-{\bf S}hapiro-{\bf V}ainshtein. It provides a strong connection between geometry of moduli spaces and the Hurwitz numbers. In practice it is a very difficult to use but it remains one of the most striking results  related to Hurwitz enumeration problem.

\section{The ELSV Formula}

Recall that  the  Hurwitz number $h_{g,\mu}$ is the number  of branched  coverings  of degree  $d$  from  smooth curves of genus $g$  to  $\PP^1$ with one  branch point (usually taken to be $\infty\in\PP^1)$ of branched type   $\mu=(\mu_1,\ldots,\mu_n)$  and $w=d+n+2g-2$ other simple branch points.

\begin{theorem}[The ELSV formula] Suppose that $g,n$ are integers $(g\geq 0, n\geq1)$ such that $2g-2+n>0$.  Let $\mu = (\mu_1 , \ldots , \mu_n)\vdash d$ and  $\Aut(\mu)$ denote the automorphism group of the  partition $\mu$.  Then,
\begin{equation}
h_{g,\mu}=\frac{w!}{|\Aut(\mu)|}\prod^n_{i=1}\frac{\mu_i^{\mu_i}}{\mu_i !}\int_{\mgnbar}\frac{1-\lambda_1+\ldots +(-1)^g\lambda_{g}}{(1-\mu_1\psi_1)\ldots (1-\mu_n\psi_n)}
\label{elsv}
\end{equation}
where $\psi_i=c_1(\LL_i)\in \mathbf H^{2i}(\mgnbar, \QQ)$ is the first Chern class of the contagent line bundle $\LL_i\longrightarrow \mgnbar$  and $\lambda_j=c_j(\mathbb E)\in \mathbf H^{2j}(\mgnbar, \QQ)$ is the $j$th Chern class of the Hodge bundle $\mathbb E\longrightarrow \mgnbar$
$$\frac{1}{1-\mu_i\psi_i}=1+\mu_1\psi_1+\ldots +\ldots \mu_i^i\psi_i^i+\ldots $$

(Observe  that the above expansion terminates because $\psi_i\in H^{2i}(\mgnbar, \QQ) $ is nilpotent.)

\end{theorem}
Notice that the ELSV  formula is a polynomial in the variables $\mu_1 , \ldots , \mu_n$.  This fact is stated in  the Golden-Jackson polynomiality conjecture \cite{GJ99-2} which this formula settles.
\begin{remark}
The ELSV formula is not applicable to coverings of genus $0$ with $1$ and $2$ marked points  since the stability condition  $2g-2+n>0$ is violated.
However,  the ELSV formula remains true for these two cases as well
\begin{equation}
\int_{\overline{\Mn}_{0, 1}}\frac{1}{(1-\mu_1\psi_1)}=\frac{1}{\mu_1^2}\;, \quad \text{and} \qquad \int_{\overline{\Mn}_{0, 2}}\frac{1}{(1-\mu_1\psi_1)(1-\mu_2\psi_2)}=\frac{1}{\mu_1+\mu_2}\; .
\end{equation}
\end{remark}

Apart from the easy  combinatorial factor, the  ELSV formula  involves the  integrals of the form

\begin{equation}
\int_{\mgnbar}\psi^{m_1}_1\ldots\psi^{m_n}\lambda_1^{k_1}\ldots \lambda_g^{k_g},
\end{equation}

called the {\bf Hodge integrals} which can be reduced to  other  integrals only involving the $\psi$-classes. The latter integral are  called {\bf descendant integrals} \cite{FP00}. The explicit evaluation of these integrals or computation of  the  intersection numbers  is a difficult task. On the other hand,  we can see that using the  ELSV formula $(\ref{elsv})$ makes it possible  to  calculate the  intersection numbers on $\mgnbar$ once the single Hurwitz numbers are known.

\subsection{ Hurwitz Formula via the ELSV formula}

Although, the ELSV formula (\ref{elsv})  is hard to use, there is a couple of very well-known cases. These cases are related to Witten conjecture \cite{Witt91} now known as  the Kontsevich's  theorem \cite {MK92}which gives a recursive relation for Hodge integrals involving $\psi$ classes only.  In return some  of Hodge integrals can be evaluated recursively through string equation and the  KdV hierarchy. In particular, we can recover the following well-known cases.

\begin{theorem}[{\bf{Hurwitz Formula}}\cite{Hur91}]

The single Hurwitz Number  formula $h_{0,\mu:}$ in  is given by,  \begin{equation}
h_{0, \mu}=\frac{(n+d-2)!}{|Aut(\mu)|} \prod_{i=1}^n\frac{\mu_i^{\mu_i}}{\mu_i!}d^{n-3}
\end{equation}
where $n+d-2$ is the number of simple branch points.
\end{theorem}
\begin{proof}
By the ELSV formula and string equation
\begin{align*}
h_{0, \mu}=&\frac{(d+n-2)!}{|Aut(\mu)|}\prod^{n}_{i=1}\frac{\mu_i^{\mu_i}}{\mu_i!}\int_{{\Mn}_{0, n}}\frac{1}{(1-\mu_1\psi_1)\ldots (1-\mu_n\psi_n)}\\
=&\frac{(d+n-2)!}{|Aut(\mu)|}\prod^{n}_{i=1}\frac{\mu_i^{\mu_i}}{\mu_i!}\sum_{m_1+\ldots+m_n=n-3}\langle \tau_{m_1}\ldots \tau_{m_n}\rangle_0\; \cdot\mu^{m_1}_1\ldots \mu_n^{m_n} \\
=&\frac{(d+n-2)!}{|Aut(\mu)|}\prod^{n}_{i=1}\frac{\mu_i^{\mu_i}}{\mu_i!}\sum_{m_1+\ldots+m_n=n-3}\frac{(n-3)!}{m_1!\ldots m_n!}\; \cdot\mu^{m_1}_1\ldots \mu_n^{m_n}\\
=&\frac{(d+n-2)!}{|Aut(\mu)|}\prod^{n}_{i=1}\frac{\mu_i^{\mu_i}}{\mu_i!} d^{n-3}.\quad\text{}
\end{align*}
\end{proof}
Moreover, we can recover  the classical formulas  of Denes, Arnol'd and Crescimano-Taylor,   cf. (\ref{JD}), (\ref{AV}) and (\ref{CT}) respectively:
\begin{corollary}[Polynomial case]
If $\mu=(d)$  then
$$h_{0, \mu}=(d-1)! \frac{d^{d}}{d!}d^{-2}=d^{d-3}.$$
\end{corollary}
\begin{corollary}[Rational  case]
If $g=0$ and   $\mu=(1^d)$  then
$$h_{0, \mu}=\frac{(2d-2)!}{d!} d^{d-3}.$$
\end{corollary}
\begin{corollary}[Arnol'd Case]
If $g=0$ and  $\mu=(\mu_1, \mu_2)\vdash d$  then
$$h_{0, \mu_2, \mu_2}= \frac{\mu_1^{\mu_1}}{\mu_1!} \cdot \frac{\mu_2^{\mu_2}}{\mu_2!}\cdot (\mu_1+\mu_2-1)!\;.$$
\end{corollary}
Another well-known case with an explicit  generating formula occur in the  computation of  genus $1$ Hurwitz numbers  $h_{1, \mu}.$ The details  can be found in \cite{GJV00}.  There has been some progress in calculation of   more generalized Hurwitz numbers. 


\section{  Single  Hurwitz Numbers and Fock space operator techniques}
Characters of the symmetric group can be easily expressed in the infinite wedge space, and so we can use Fock space techniques to study the Hurwitz numbers.
Fock space techniques arose in physics, and were introduced  into  Hurwitz theory  in  \cite{OP06}.  Translating Hurwitz enumeration problem into a question of operators on the Fock space gives access to the structure of generating functions for enumerative invariants. Let  $V$ be a vector space over $\CC$ with a basis indexed by half-integers 
 $$V=\bigoplus_{i\in  \Z+\frac12} \CC\cdot\underline {i}.$$

Write $\ZZ_{1/2}^+$ for  the positive half integers, and $\ZZ_{1/2}^-$  for the negative half integers. A state $S$ is a subset of half integers $S=\{s_{1}<s_{2}<\ldots\}\subset \ZZ+\frac12$ such that  both   $S \backslash \ZZ_{1/2}^-$ and $\ZZ_{1/2}^-\backslash S$  are finite. The {\it fermionic Fock space } is the vector space  
$$\infwedge V=\bigoplus \CC v_S$$
with a basis $\{ v_S\}$ spanned by all formal symbols in the wedge product 
$$v_{S}=\underline{s_{1}}\wedge\underline{s_{2}}\wedge\cdots.$$ 

Denote by $ (. \ ,\ .) $ the unique inner product on $\infwedge V$ for which our basis $\{ v_S\}$ is orthonormal. Observe that, the wedge product is associative, bilinear, and anticommutative, that is $a\wedge b=-b\wedge a$ for any half integers $a, b$. Assign an integer $c$ to  each semi-infinite wedge  $v_S$ of $\infwedge V$  called the \emph{charge} of $S$ defined by 
$$c= |S\cap \Z^-_{1/2}|-|S\cap \Z^+_{1/2}|.$$

Let $\infwedge_c V$ denotes the subspace generated by semi-infinite wedges of charge $c$. Then the  Fock space is decomposable by the charge:

$$\infwedge V=\bigoplus_{c\in \ZZ} \infwedge_c V.$$

We will mostly be concerned with the   {\it charge zero subspace of Fock space}  $\infwedge_0V\subset \infwedge V$ --
 which is the subspace spanned by all basis elements with charge $0$.  Remarkably, the charge zero subspace  has a basis indexed by integer partitions  $\lambda= (\lambda_{1}, \ldots,\lambda_{\ell(\lambda)})$ of all integers $\mathcal{P}$: 
\begin{align*}v_S:=v_\lambda=\underline{\lambda_1-1+\tfrac{1}{2}}\wedge\underline{\lambda_2-2+\tfrac{1}{2}}\wedge\underline{\lambda_3-3+\tfrac{1}{2}}\wedge\cdots \underline{\lambda_i-i+\frac12}\cdots .\end{align*} 
Note that  the state $S$  is given by  $S=\{\lambda_{1}-\frac{1}{2}, \lambda_{2}-\frac{3}{2}, \cdots \lambda_i-i+\frac12 \cdots\}$ for some unique integer partition $\lambda$. 

\subsection{ Representation of  a state in a Maya diagram}
A useful way to represent a state or rather a basis element of the  Fock Space is through a Maya diagram: a sequence of circles at centered at $\ZZ+\tfrac12$  on the real line, with the positive entries going to the left and the negative entries to the right.   A black bead is placed at each position $i$ where the corresponding vector $s_i$ appears in the wedge of the entries of $S$.  For example, the partition  $\lambda=(4,3,1,1)$ corresponds to the state $S=\{\frac{7}{2}, \frac{3}{2},  -\frac{3}{2}, -\frac{5}{2}, -\frac{9}{2},  \cdots\}$, the Maya diagram and  wedge product below.

\begin{center}
\begin{tikzpicture}[scale=.6, yshift=-.0cm]
\draw[-, very thick, color=red!50] (0, -.5) -- (0,-1.5);
\draw (-6.5,-1)  circle(0.25)  node[below=3pt]{$$};
\draw (-5.5,-1)  circle(0.25)  node[below=3pt]{$\frac{11}{2}$};
\draw (-4.5,-1)  circle(0.25) node[below=3pt]{$\frac{9}{2}$};
\filldraw (-3.5,-1)  circle(0.25) node[below=3pt]{$\frac{7}{2}$};
\draw (-2.5,-1)  circle(0.25)node[below=3pt]{$\frac{5}{2}$};
\filldraw (-1.5,-1)  circle(0.25)node[below=3pt]{$\frac{3}{2}$};
\draw (-.5,-1)  circle(0.25) node[below=3pt]{$\frac{1}{2}$};
\draw (.5,-1) circle(0.25)node[below=3pt]{$\frac{-1}{2}$};
\filldraw (1.5,-1)  circle(0.25) node[below=3pt]{$\frac{-3}{2}$};
\filldraw (2.5,-1)  circle(0.25) node[below=3pt]{$\frac{-5}{2}$};
\draw (3.5,-1) circle(0.25)node[below=3pt]{$\frac{-7}{2}$};
\filldraw (4.5,-1)  circle(0.25) node[below=3pt]{$\frac{-9}{2}$};
\filldraw (5.5,-1)  circle(0.25) node[below=3pt]{$$};
\draw [-] (-8.5, -1) -- (8.5,-1);

\node[ ] at (-0.5, -3.5) {$v_\lambda=v_{(4,3,1,1)}=\underline{\tfrac72}\wedge\underline{\tfrac32} \wedge\underline{-\tfrac32}\wedge \underline{-\tfrac52}\wedge
\underline{-\tfrac{9}2}\wedge\underline{-\tfrac{11}2}\wedge\cdots$};
\end{tikzpicture}
\end{center}

This notation  is a gateway to an intuitive bijection between partitions $\mathcal{P}$ and basis elements  of the $0$-charge subspace $\infwedge_0V$. 
To see this, draw a partitions rotated $\pi/4$ radians counterclockwise and scaled up by a factor of $\sqrt{2}$, so that each segment of the border path of $\lambda$ is centered above a half integer on the $x$-axis, with origin above the square $0$. Placing a black bead for every line segment in $\lambda$ in the direction $(1,1)$ (an upstep)  above each half integer $s\in S$.  For instance, the partition  $\lambda=(4,3,1,1)$ corresponds to the Maya diagram shown.
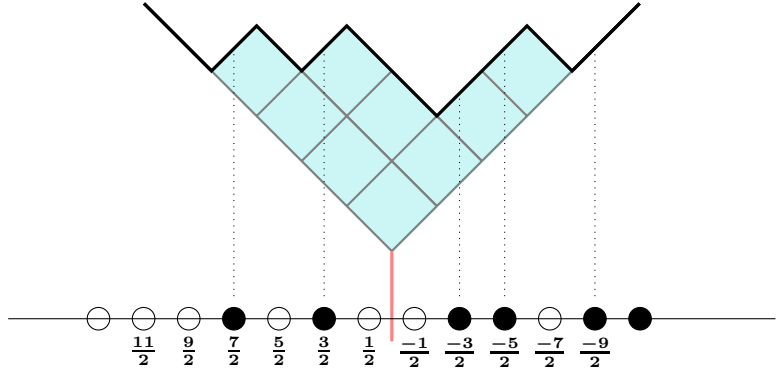
\begin{center}
\begin{tikzpicture}[
every node/.style = {
font=\footnotesize
},
]

\begin{scope}[rotate=45, very thick, scale=.6*1.412]
\draw (0,5.5) -- (0, 4) -- (1,4)--(1,3) -- (2,3) -- (2,1) -- (4,1) -- (4,0)-- (5.5,0);
\draw[fill, q1!20] (0,0)--(0, 4) -- (1,4)--(1,3) -- (2,3) -- (2,1) -- (4,1) -- (4,0)--cycle;
\end{scope}

\begin{scope}[gray, thick, scale=.6]
\clip (-5.5, 5.5) rectangle (5.5, -5);
\draw[rotate=45, scale=1.412] (0,0) grid (1,4);
\draw[rotate=45, scale=1.412] (1,0) grid (2,3);
\draw[rotate=45, scale=1.412] (2,0) grid (4,1);
\end{scope}

\begin{scope}[scale=.6, dotted]

\draw (-3.5,-1) -- (-3.5, 4.5);
\draw (-1.5,-1) -- (-1.5, 4.5);
\draw (1.5,-1) -- (1.5, 3.5);
\draw (2.5,-1) -- (2.5, 4.5);
\draw (4.5,-1) -- (4.5, 4.5);
\end{scope}

\begin{scope}[rotate=45, very thick, scale=.6*1.412]
\draw (0,5.5) -- (0, 4) -- (1,4)--(1,3) -- (2,3) -- (2,1) -- (4,1) -- (4,0)-- (5.5,0);
\end{scope}

\begin{scope}[scale=.6, yshift=-.5cm]
\draw[-, very thick, color=red!50] (0,-1.5) -- (0,.5);
\draw (-6.5,-1)  circle(0.25)  node[below=3pt]{$$};
\draw (-5.5,-1)  circle(0.25)  node[below=3pt]{$\frac{11}{2}$};
\draw (-4.5,-1)  circle(0.25) node[below=3pt]{$\frac{9}{2}$};
\filldraw (-3.5,-1)  circle(0.25) node[below=3pt]{$\frac{7}{2}$};
\draw (-2.5,-1)  circle(0.25)node[below=3pt]{$\frac{5}{2}$};
\filldraw (-1.5,-1)  circle(0.25)node[below=3pt]{$\frac{3}{2}$};
\draw (-.5,-1)  circle(0.25) node[below=3pt]{$\frac{1}{2}$};
\draw (.5,-1) circle(0.25)node[below=3pt]{$\frac{-1}{2}$};
\filldraw (1.5,-1)  circle(0.25) node[below=3pt]{$\frac{-3}{2}$};
\filldraw (2.5,-1)  circle(0.25) node[below=3pt]{$\frac{-5}{2}$};
\draw (3.5,-1) circle(0.25)node[below=3pt]{$\frac{-7}{2}$};
\filldraw (4.5,-1)  circle(0.25) node[below=3pt]{$\frac{-9}{2}$};
\filldraw (5.5,-1)  circle(0.25) node[below=3pt]{$$};
\draw [-] (-8.5, -1) -- (8.5,-1);
\end{scope}
 \end{tikzpicture}
\end{center}

The charge $0$ state corresponds to the empty partition represented by a special vector $v_\varnothing$ called the  {\it vacuum vector},  
$$
 v_\varnothing=\underline{-\tfrac12}\wedge\underline{-\tfrac32}\wedge \underline{-\tfrac 52}\wedge \underline{-\tfrac{7}2}\wedge\underline{-\tfrac{9}2}\cdots.
$$
 Namely,  the Maya diagram for the vacuum vector has a black bead for every negative half-integer.

\subsection{ Operators on the fermionic Fock spaces}
We now define interesting natural operators among Fock spaces by their action on the basis element.  

\begin{enumerate}[I).]
\item Basic operators on Fock spaces $\infwedge_c V$ for any charge $c$.

\begin{enumerate}
\item The {\bf wedging operator}  $\psi_{k} : \infwedge_c V\longrightarrow \infwedge_{c+1} V$ indexed by the half integers is defined by 
\[\psi_k(v_S)=\underline k\wedge v_S=\begin{cases}
0 ,&  k\in S\\[2mm]
\pm v_{S\cup\{k\}} , & k\not\in S.
\end{cases}\]
The sign is obtained by  applying the anticommutativity of the wedge product until the sequence is decreasing.
For example,
\begin{align*}
\psi_{\frac32}(v_{(3,1)})&=\underline{\tfrac32}\wedge \left(\underline{\tfrac52}\wedge\underline{-\tfrac12}\wedge \underline{-\tfrac52}\wedge\underline{-\tfrac72}\wedge\cdots \right)\\
&=- \left(\underline{\tfrac52}\wedge\underline{\tfrac32}\wedge\underline{-\tfrac12}
\wedge\underline{-\tfrac52}\wedge\underline{-\tfrac72}\wedge\cdots\right)\\
&=-v_{(2,2,1)}.
\end{align*}

\item The {\bf contracting operator} $\psi_{k}^{*}: \infwedge_{c+1} V\longrightarrow \infwedge_{c} V$ is the adjoint of $\psi_{k}$ with respect to our inner product  given by 
\[\psi_k^{*}(v_S)=\underline k\wedge v_S=\begin{cases}
\pm v_{S\cup\{k\}} , & k\in S\\[2mm]
0 ,&  k\not\in S.
\end{cases}\]

These operators satisfy the anti-commutation relations:
\[\psi_i\psi_j^*+\psi_j^*\psi_i=\delta_{ij},\quad  \psi_i\psi_j +\psi_j\psi_i=0, \quad \psi_i^*\psi_j^*+\psi_j^*\psi_i^*=0\]

\end{enumerate}
The operator $\psi_{k}$ increases the charge by $1$, and $\psi_{k}^{*}$  reduces the charge by $1$, we can apply them in sequence and preserve charge which yield interesting  operators on $\infwedge_0V$.
\item  To keep track of the convergence of infinite sums of products of  the wedging and contraction operators, we define the {\bf normally ordered products} of  $\psi_i$ and $\psi^*_i$ by
\[:\psi_i\psi_j^*:=\left\{\begin{array}{ll}\psi_i\psi^*_j,&j>0\\-\psi^*_j\psi_i,&j<0.\end{array}\right.\]

\begin{enumerate}
\item  The operator $E_{i, j}:=:\psi_i\psi_j^*:$  is attempting to move a bead from position $j$ to position $i$ if it is possible. The normal ordering  puts  a minus sign if $j$ is negative.  \\[2mm]

If $i<j$, the result of applying $E_{i, j}$  to a vector $v_\lambda$, where $\lambda$ is some partition, is to remove a ribbon  to $\lambda$ if possible and a sign is added according to the parity of the height of the rim; if the ribbon  cannot be removed, the result is zero. For example,
\[ E_{-\frac12, \frac72} \big (v_{(4,3,1)}\big )=-v_{(2,1,1)}.\]
This is illustrated in the following diagram below.

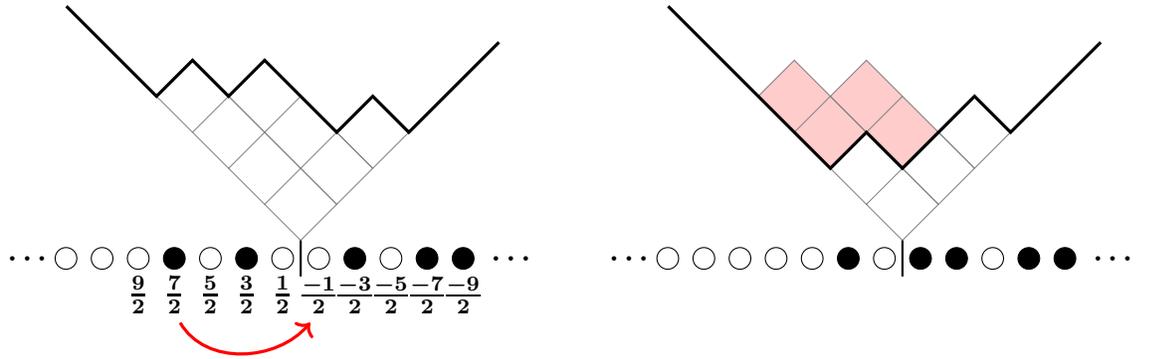
\begin{center}
\begin{tikzpicture}[scale=.8]

\begin{scope}[gray, very thin, scale=.6]
\clip (-7.5, 7.5) rectangle (7.5, -1);
\draw[rotate=45, scale=1.412] (0,0) grid (1,4);
\draw[rotate=45, scale=1.412] (1,0) grid (2,3);
\draw[rotate=45, scale=1.412] (2,0) grid (3,1);

\end{scope}
\begin{scope}[rotate=45, very thick, scale=.6*1.412]
\draw (0,6.5) -- (0, 4) -- (1,4) -- (1,3) -- (2,3) -- (2,1) -- (3,1)-- (3,0)-- (5.5,0);
\end{scope}
\begin{scope}[scale=.6, yshift=-.5cm]

\draw (-7.5,0) node{$\cdots$};
\draw (-6.5,0) circle (.3)node[below=3pt]{$$};
\draw (-5.5,0) circle (.3)node[below=3pt]{$$};
\draw (-4.5,0) circle (.3)node[below=3pt]{$\frac{9}{2}$};
\filldraw (-3.5,0) circle (.3)node[below=3pt]{$\frac{7}{2}$};
\draw (-2.5,0) circle (.3)node[below=3pt]{$\frac{5}{2}$};
\filldraw (-1.5,0) circle (.3) node[below=3pt]{$\frac{3}{2}$};
\draw (-.5,0) circle (.3) node[below=3pt]{$\frac{1}{2}$};
\draw[-,  thick] (0,-.5) -- (0,.5); 
\draw (.5,0) circle (.3) node[below=3pt]{$\frac{-1}{2}$};
\filldraw (1.5,0) circle (.3)node[below=3pt]{$\frac{-3}{2}$};
\draw (2.5,0) circle (.3)node[below=3pt]{$\frac{-5}{2}$};
\filldraw (3.5,0) circle (.3)node[below=3pt]{$\frac{-7}{2}$};
\filldraw (4.5,0) circle (.3)node[below=3pt]{$\frac{-9}{2}$};
\draw (5.5,0) node{\quad $\cdots$};
\node (start) at (-3.5, -1.5) {};
\node(finish) at (0.5, -1.5) {};
\draw[->, very thick, color=red] (start) to [out=-60, in=230] (finish);
\end{scope}

\begin{scope}[xshift=10cm]
\begin{scope}[gray, very thin, scale=.6]
\clip (-7.5, 7.5) rectangle (7.5, -1);
\draw[rotate=45, scale=1.412, fill, red!20] (0, 4) -- (1,4) -- (1,3) -- (2,3) -- (2,1)--(1,1)--(1,2)--(0,2)--cycle;
\draw[rotate=45, scale=1.412] (0,0) grid (1,4);
\draw[rotate=45, scale=1.412] (1,0) grid (2,3);
\draw[rotate=45, scale=1.412] (2,0) grid (3,1);
\end{scope}
\begin{scope}[rotate=45, very thick, scale=.6*1.412]
\draw (0,6.5) --(0,2) -- (1, 2) -- (1,1) --(3,1 ) --(3,0) -- (5.5,0);
\end{scope}
\begin{scope}[scale=.6, yshift=-.5cm]
\draw (-7.5,0) node{$\cdots$};
\draw (-6.5,0) circle (.3);
\draw (-5.5,0) circle (.3);
\draw (-4.5,0) circle (.3);
\draw (-3.5,0) circle (.3);
\draw (-2.5,0) circle (.3);
\filldraw (-1.5,0) circle (.3);
\draw (-.5,0) circle (.3);
\draw[-,  thick] (0,-.5) -- (0,.5); 
\filldraw (.5,0) circle (.3);
\filldraw (1.5,0) circle (.3);
\draw (2.5,0) circle (.3);
\filldraw (3.5,0) circle (.3);
\filldraw (4.5,0) circle (.3);
\draw (5.5,0) node{\quad $\cdots$};
\end{scope}

\end{scope}
\end{tikzpicture}
\end{center}

If $i>j$, the result of applying $E_{i, j}$ to a vector $v_\lambda$ is to add a ribbon  to $\lambda$ if possible, and to add a sign according to the parity of the height of the rim.  If the addition of the ribbon  is not possible, the result is zero.\\[2mm]

The case when $i = j$, anti-commutation relations and  the normal ordering of the product takes effect. 

\item The {\bf bosonic operator} $\alpha_{n}$. These operators are constructed from the fermionic operators as follows.
$$
\alpha_{n}:=\sum_{k\in \Z+\frac12 }E_{k-n,k}
$$
If  $n>0$, the operator  attempts to remove ribbons of length $n$ from the Mya diagram of $\lambda$ in $v_{\lambda}$. If there are multiple ribbons that can be removed, the operator returns the sum of all contributions, weighted by this sign. The sign is  $(-1)^{h-1}$, where $h$ is the height $h$ of the ribbon removed and it is defined as the number of rows it occupies.\\[2mm]

For example, it follows that $\alpha_{3}v_{(5,4,3)}=v_{(3,1,1)}-V_{(5,2,2)}-V_{(5,4)}$. This is be illustrated  using Maya diagrams below

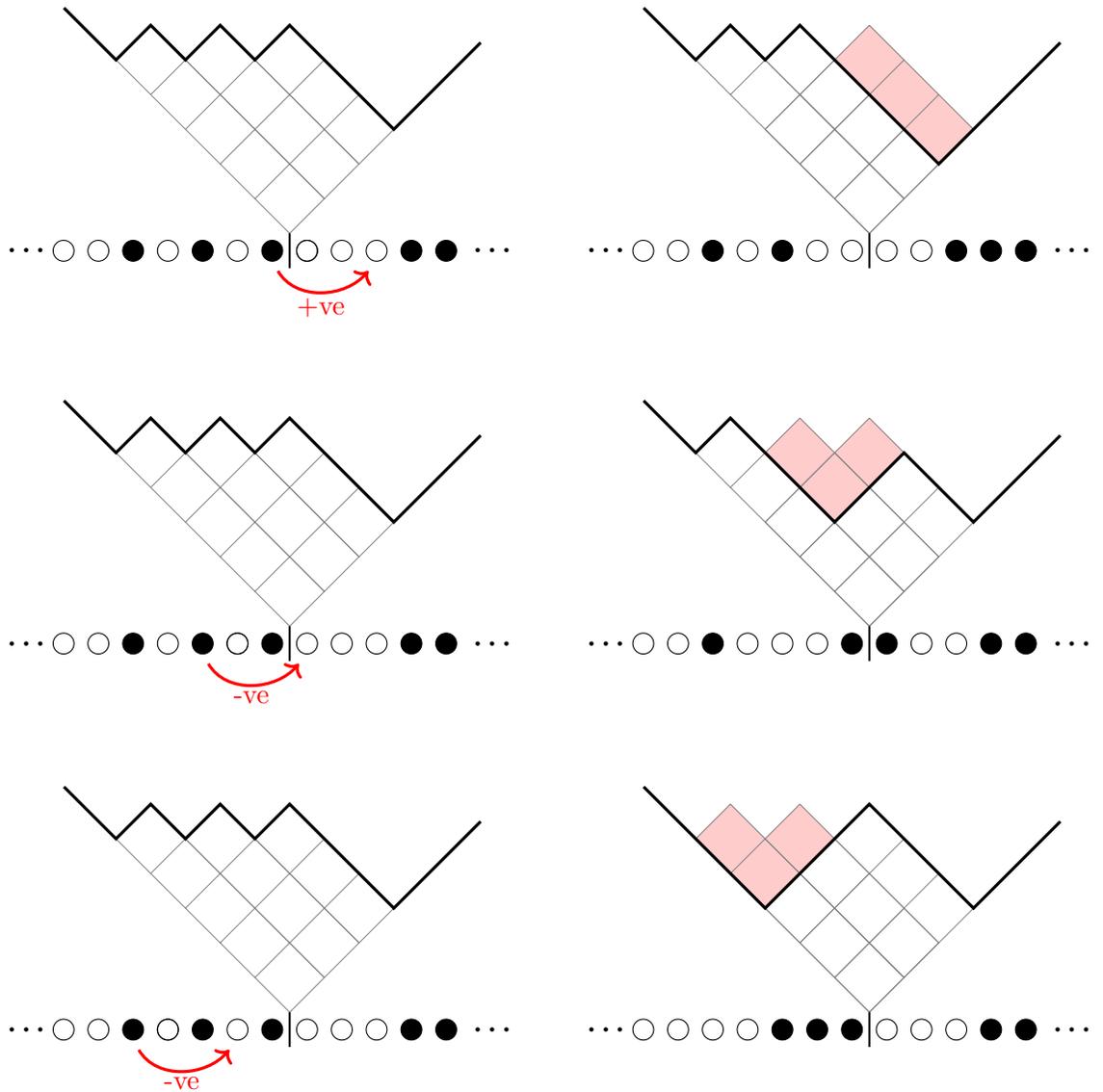
\begin{center}
\begin{tikzpicture}[scale=.8]

\begin{scope}[gray, very thin, scale=.6]
\clip (-7.5, 7.5) rectangle (7.5, -1);
\draw[rotate=45, scale=1.412] (0,0) grid (1,5);
\draw[rotate=45, scale=1.412] (1,0) grid (2,4);
\draw[rotate=45, scale=1.412] (2,0) grid (3,3);

\end{scope}
\begin{scope}[rotate=45, very thick, scale=.6*1.412]
\draw (0,6.5) -- (0, 5) -- (1,5) -- (1,4) -- (2,4) -- (2,3) -- (3,3)-- (3,0)-- (5.5,0);
\end{scope}
\begin{scope}[scale=.6, yshift=-.5cm]

\draw (-7.5,0) node{$\cdots$};
\draw (-6.5,0) circle (.3);
\draw (-5.5,0) circle (.3);
\filldraw (-4.5,0) circle (.3);
\draw (-3.5,0) circle (.3);
\filldraw (-2.5,0) circle (.3);
\draw (-1.5,0) circle (.3);
\filldraw (-.5,0) circle (.3);
\draw[-,  thick] (0,-.5) -- (0,.5); 
\draw (.5,0) circle (.3);
\draw (1.5,0) circle (.3);
\draw (2.5,0) circle (.3);
\filldraw (3.5,0) circle (.3);
\filldraw (4.5,0) circle (.3);
\draw (5.5,0) node{\quad $\cdots$};
\draw (0.5,0) circle (.3) node[below=15pt, red]{\quad +ve} ;
\node (start) at (-0.5, -.3) {};
\node(finish) at (2.5, -.3) {};
\draw[->, very thick, color=red] (start) to [out=-60, in=230] (finish);
\end{scope}

\begin{scope}[xshift=10cm]
\begin{scope}[gray, very thin, scale=.6]
\clip (-7.5, 7.5) rectangle (7.5, -1);
\draw[rotate=45, scale=1.412, fill, red!20]  (2,3) -- (3,3) -- (3,0)--(2,0)--cycle;
\draw[rotate=45, scale=1.412] (0,0) grid (1,5);
\draw[rotate=45, scale=1.412] (1,0) grid (2,4);
\draw[rotate=45, scale=1.412] (0,0) grid (3,3);
\end{scope}
\begin{scope}[rotate=45, very thick, scale=.6*1.412]
\draw (0,6.5) -- (0, 5) -- (1,5) -- (1,4) -- (2,4) -- (2,0)-- (5.5,0);
\end{scope}
\begin{scope}[scale=.6, yshift=-.5cm]
\draw (-7.5,0) node{$\cdots$};
\draw (-6.5,0) circle (.3);
\draw (-5.5,0) circle (.3);
\filldraw (-4.5,0) circle (.3);
\draw (-3.5,0) circle (.3);
\filldraw (-2.5,0) circle (.3);
\draw (-1.5,0) circle (.3);
\draw (-.5,0) circle (.3);
\draw[-,  thick] (0,-.5) -- (0,.5); 
\draw (.5,0) circle (.3);
\draw (1.5,0) circle (.3);
\filldraw (2.5,0) circle (.3);
\filldraw (3.5,0) circle (.3);
\filldraw (4.5,0) circle (.3);
\draw (5.5,0) node{\quad $\cdots$};
\end{scope}

\end{scope}
\end{tikzpicture}
\end{center}

\begin{center}
\begin{tikzpicture}[scale=.8]

\begin{scope}[gray, very thin, scale=.6]
\clip (-7.5, 7.5) rectangle (7.5, -1);
\draw[rotate=45, scale=1.412] (0,0) grid (1,5);
\draw[rotate=45, scale=1.412] (1,0) grid (2,4);
\draw[rotate=45, scale=1.412] (2,0) grid (3,3);

\end{scope}
\begin{scope}[rotate=45, very thick, scale=.6*1.412]
\draw (0,6.5) -- (0, 5) -- (1,5) -- (1,4) -- (2,4) -- (2,3) -- (3,3)-- (3,0)-- (5.5,0);
\end{scope}
\begin{scope}[scale=.6, yshift=-.5cm]

\draw (-7.5,0) node{$\cdots$};
\draw (-6.5,0) circle (.3);
\draw (-5.5,0) circle (.3);
\filldraw (-4.5,0) circle (.3);
\draw (-3.5,0) circle (.3);
\filldraw (-2.5,0) circle (.3);
\draw (-1.5,0) circle (.3);
\filldraw (-.5,0) circle (.3);
\draw[-,  thick] (0,-.5) -- (0,.5); 
\draw (.5,0) circle (.3);
\draw (1.5,0) circle (.3);
\draw (2.5,0) circle (.3);
\filldraw (3.5,0) circle (.3);
\filldraw (4.5,0) circle (.3);
\draw (5.5,0) node{\quad $\cdots$};
\draw (-1.5,0) circle (.3) node[below=15pt, red]{\quad -ve} ;
\node (start) at (-2.5, -.3) {};
\node(finish) at (0.5, -.3) {};
\draw[->, very thick, color=red] (start) to [out=-60, in=230] (finish);
\end{scope}

\begin{scope}[xshift=10cm]
\begin{scope}[gray, very thin, scale=.6]
\clip (-7.5, 7.5) rectangle (7.5, -1);
\draw[rotate=45, scale=1.412, fill, red!20] (1, 4) -- (2,4) -- (2,3) -- (3,3) -- (3,2)--(1,2)--cycle;
\draw[rotate=45, scale=1.412] (0,0) grid (1,5);
\draw[rotate=45, scale=1.412] (1,0) grid (2,4);
\draw[rotate=45, scale=1.412] (0,0) grid (3,3);
\end{scope}
\begin{scope}[rotate=45, very thick, scale=.6*1.412]
\draw (0,6.5) -- (0, 5) -- (1,5) -- (1,2) -- (3,2) -- (3,0)-- (5.5,0);
\end{scope}
\begin{scope}[scale=.6, yshift=-.5cm]
\draw (-7.5,0) node{$\cdots$};
\draw (-6.5,0) circle (.3);
\draw (-5.5,0) circle (.3);
\filldraw (-4.5,0) circle (.3);
\draw (-3.5,0) circle (.3);
\draw (-2.5,0) circle (.3);
\draw (-1.5,0) circle (.3);
\filldraw (-.5,0) circle (.3);
\draw[-,  thick] (0,-.5) -- (0,.5); 
\filldraw (.5,0) circle (.3);
\draw (1.5,0) circle (.3);
\draw (2.5,0) circle (.3);
\filldraw (3.5,0) circle (.3);
\filldraw (4.5,0) circle (.3);
\draw (5.5,0) node{\quad $\cdots$};
\end{scope}

\end{scope}
\end{tikzpicture}
\end{center}

\begin{center}
\begin{tikzpicture}[scale=.8]

\begin{scope}[gray, very thin, scale=.6]
\clip (-7.5, 7.5) rectangle (7.5, -1);
\draw[rotate=45, scale=1.412] (0,0) grid (1,5);
\draw[rotate=45, scale=1.412] (1,0) grid (2,4);
\draw[rotate=45, scale=1.412] (2,0) grid (3,3);

\end{scope}
\begin{scope}[rotate=45, very thick, scale=.6*1.412]
\draw (0,6.5) -- (0, 5) -- (1,5) -- (1,4) -- (2,4) -- (2,3) -- (3,3)-- (3,0)-- (5.5,0);
\end{scope}
\begin{scope}[scale=.6, yshift=-.5cm]

\draw (-7.5,0) node{$\cdots$};
\draw (-6.5,0) circle (.3);
\draw (-5.5,0) circle (.3);
\filldraw (-4.5,0) circle (.3);
\draw (-3.5,0) circle (.3);
\filldraw (-2.5,0) circle (.3);
\draw (-1.5,0) circle (.3);
\filldraw (-.5,0) circle (.3);
\draw[-,  thick] (0,-.5) -- (0,.5); 
\draw (.5,0) circle (.3);
\draw (1.5,0) circle (.3);
\draw (2.5,0) circle (.3);
\filldraw (3.5,0) circle (.3);
\filldraw (4.5,0) circle (.3);
\draw (5.5,0) node{\quad $\cdots$};
\draw (-3.5,0) circle (.3) node[below=15pt, red]{\quad -ve} ;
\node (start) at (-4.5, -.3) {};
\node(finish) at (-1.5, -.3) {};
\draw[->, very thick, color=red] (start) to [out=-60, in=230] (finish);
\end{scope}

\begin{scope}[xshift=10cm]
\begin{scope}[gray, very thin, scale=.6]
\clip (-7.5, 7.5) rectangle (7.5, -1);
\draw[rotate=45, scale=1.412, fill, red!20] (0, 5) -- (1,5) -- (1,4) -- (2,4) -- (2,3)--(0,3)--cycle;
\draw[rotate=45, scale=1.412] (0,0) grid (1,5);
\draw[rotate=45, scale=1.412] (1,0) grid (2,4);
\draw[rotate=45, scale=1.412] (0,0) grid (3,3);
\end{scope}
\begin{scope}[rotate=45, very thick, scale=.6*1.412]
\draw (0,6.5) --(0,3) -- (3, 3) -- (3,0)  -- (5.5,0);
\end{scope}
\begin{scope}[scale=.6, yshift=-.5cm]
\draw (-7.5,0) node{$\cdots$};
\draw (-6.5,0) circle (.3);
\draw (-5.5,0) circle (.3);
\draw (-4.5,0) circle (.3);
\draw (-3.5,0) circle (.3);
\filldraw (-2.5,0) circle (.3);
\filldraw (-1.5,0) circle (.3);
\filldraw (-.5,0) circle (.3);
\draw[-,  thick] (0,-.5) -- (0,.5); 
\draw (.5,0) circle (.3);
\draw (1.5,0) circle (.3);
\draw (2.5,0) circle (.3);
\filldraw (3.5,0) circle (.3);
\filldraw (4.5,0) circle (.3);
\draw (5.5,0) node{\quad $\cdots$};
\end{scope}

\end{scope}
\end{tikzpicture}
\end{center}

If  $n<0$, the operator acts in a similar way It attempts to move downsteps $n$ places to the right, which graphically corresponds to adding a ribbon of length $n$. The sign can be calculated in the same way from the height of the added ribbon.\\[2mm]
For example, consider the effect of $\alpha_{-3}$ on $v_{(3,1)}$. We depict the action in the following diagrams:

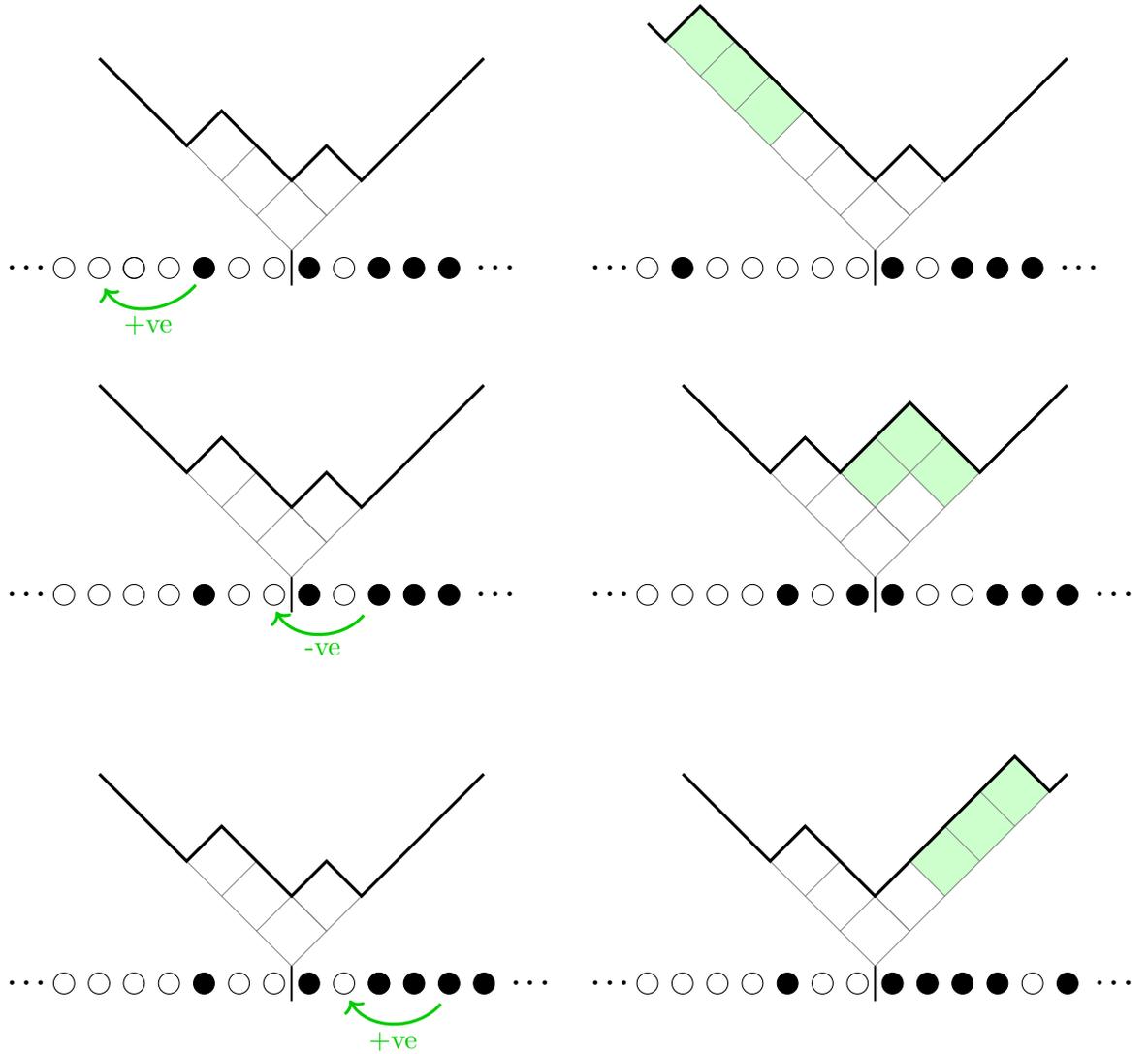
\begin{center}
\begin{tikzpicture}[scale=.8]

\begin{scope}[gray, very thin, scale=.6]
\clip (-5.5, 5.5) rectangle (5.5, -1);
\draw[rotate=45, scale=1.412] (0,0) grid (1,3);
\draw[rotate=45, scale=1.412] (1,0) grid (2,1);

\end{scope}
\begin{scope}[rotate=45, very thick, scale=.6*1.412]
\draw (0,5.5) -- (0, 3) -- (1,3) -- (1,1) -- (2,1) -- (2,0) -- (5.5,0);
\end{scope}
\begin{scope}[scale=.6, yshift=-.5cm]
\draw[-,  thick] (0,-.5) -- (0,.5);
\draw (-7.5,0) node{$\cdots$};
\draw (-6.5,0) circle (.3);
\draw (-5.5,0) circle (.3);
\draw (-4.5,0) circle (.3);
\draw (-3.5,0) circle (.3);
\filldraw (-2.5,0) circle (.3);
\draw (-1.5,0) circle (.3);
\draw (-.5,0) circle (.3);
\filldraw (.5,0) circle (.3);
\draw (1.5,0) circle (.3);
\filldraw (2.5,0) circle (.3);
\filldraw (3.5,0) circle (.3);
\filldraw (4.5,0) circle (.3);
\draw (5.5,0) node{\quad $\cdots$};
\draw (-4.5,0) circle (.3) node[below=15pt, green!80!black]{\quad +ve} ;
\node (start) at (-5.5, -.3) {};
\node(finish) at (-2.5, -.2) {};
\draw[<-, very thick, color=green!80!black] (start) to [out=-60, in=230] (finish);
\end{scope}

\begin{scope}[xshift=10cm]
\begin{scope}[gray, very thin, scale=.6]
\clip (-7.5, 7.5) rectangle (7.5, -1);
\draw[rotate=45, scale=1.412, fill, green!20] (0,3)--(1,3)--(1,6)--(0,6)--cycle;
\draw[rotate=45, scale=1.412] (0,0) grid (1,6);
\draw[rotate=45, scale=1.412] (1,0) grid (2,1);
\end{scope}
\begin{scope}[rotate=45, very thick, scale=.6*1.412]
\draw (0,6.5) --(0,6) -- (1, 6) -- (1,1) -- (2,1) -- (2,0) -- (2,0) -- (5.5,0);
\end{scope}
\begin{scope}[scale=.6, yshift=-.5cm]
\draw[-,  thick] (0,-.5) -- (0,.5);
\draw (-7.5,0) node{$\cdots$};
\draw (-6.5,0) circle (.3);
\filldraw (-5.5,0) circle (.3);
\draw (-4.5,0) circle (.3);
\draw (-3.5,0) circle (.3);
\draw (-2.5,0) circle (.3);
\draw (-1.5,0) circle (.3);
\draw (-.5,0) circle (.3);
\filldraw (.5,0) circle (.3);
\draw (1.5,0) circle (.3);
\filldraw (2.5,0) circle (.3);
\filldraw (3.5,0) circle (.3);
\filldraw (4.5,0) circle (.3);
\draw (5.5,0) node{\quad $\cdots$};
\end{scope}

\end{scope}
\end{tikzpicture}
\end{center}

\begin{center}
\begin{tikzpicture}[scale=.8]

\begin{scope}[gray, very thin, scale=.6]
\clip (-5.5, 5.5) rectangle (5.5, -1);
\draw[rotate=45, scale=1.412] (0,0) grid (1,3);
\draw[rotate=45, scale=1.412] (1,0) grid (2,1);

\end{scope}
\begin{scope}[rotate=45, very thick, scale=.6*1.412]
\draw (0,5.5) -- (0, 3) -- (1,3) -- (1,1) -- (2,1) -- (2,0) -- (5.5,0);
\end{scope}
\begin{scope}[scale=.6, yshift=-.5cm]
\draw[-,  thick] (0,-.5) -- (0,.5);
\draw (-7.5,0) node{$\cdots$};
\draw (-6.5,0) circle (.3);
\draw (-5.5,0) circle (.3);
\draw (-4.5,0) circle (.3);
\draw (-3.5,0) circle (.3);
\filldraw (-2.5,0) circle (.3);
\draw (-1.5,0) circle (.3);
\draw (-.5,0) circle (.3);
\filldraw (.5,0) circle (.3);
\draw (1.5,0) circle (.3);
\filldraw (2.5,0) circle (.3);
\filldraw (3.5,0) circle (.3);
\filldraw (4.5,0) circle (.3);
\draw (5.5,0) node{\quad $\cdots$};
\filldraw (.5,0) circle (.3) node[below=15pt, green!80!black]{\quad-ve} ;
\node (start) at (-.6, -.2) {};
\node(finish) at (2.3, -.3) {};
\draw[<-, very thick, color=green!80!black] (start) to [out=-60, in=230] (finish);
\end{scope}

\begin{scope}[xshift=10cm]
\begin{scope}[gray, very thin, scale=.6]
\clip (-5.5, 5.5) rectangle (5.5, -1);
\draw[rotate=45, scale=1.412, fill, green!20] (1,1)--(1,2)--(3,2)--(3,0)--(2,0)--(2,1)--cycle;
\draw[rotate=45, scale=1.412] (0,0) grid (1,3);
\draw[rotate=45, scale=1.412] (1,0) grid (3,2);
\end{scope}
\begin{scope}[rotate=45, very thick, scale=.6*1.412]
\draw (0,5.5) -- (0, 3) -- (1,3) -- (1,2) -- (3,2) -- (3,0) -- (5.5,0);
\end{scope}
\begin{scope}[scale=.6, yshift=-.5cm]
\draw[-,  thick] (0,-.5) -- (0,.5);
\draw (-7.5,0) node{$\cdots$};
\draw (-6.5,0) circle (.3);
\draw (-5.5,0) circle (.3);
\draw (-4.5,0) circle (.3);
\draw (-3.5,0) circle (.3);
\filldraw (-2.5,0) circle (.3);
\draw (-1.5,0) circle (.3);
\filldraw (-.5,0) circle (.3);
\filldraw (.5,0) circle (.3);
\draw (1.5,0) circle (.3);
\draw (2.5,0) circle (.3);
\filldraw (3.5,0) circle (.3);
\filldraw (4.5,0) circle (.3);
\filldraw (5.5,0) circle (.3);
\draw (6.5,0) node{\quad $\cdots$};
\end{scope}

\end{scope}
\end{tikzpicture}
\end{center}

\begin{center}
\begin{tikzpicture}[scale=.8]

\begin{scope}[gray, very thin, scale=.6]
\clip (-5.5, 5.5) rectangle (5.5, -1);
\draw[rotate=45, scale=1.412] (0,0) grid (1,3);
\draw[rotate=45, scale=1.412] (1,0) grid (2,1);

\end{scope}
\begin{scope}[rotate=45, very thick, scale=.6*1.412]
\draw (0,5.5) -- (0, 3) -- (1,3) -- (1,1) -- (2,1) -- (2,0) -- (5.5,0);
\end{scope}
\begin{scope}[scale=.6, yshift=-.5cm]
\draw[-,  thick] (0,-.5) -- (0,.5);
\draw (-7.5,0) node{$\cdots$};
\draw (-6.5,0) circle (.3);
\draw (-5.5,0) circle (.3);
\draw (-4.5,0) circle (.3);
\draw (-3.5,0) circle (.3);
\filldraw (-2.5,0) circle (.3);
\draw (-1.5,0) circle (.3);
\draw (-.5,0) circle (.3);
\filldraw (.5,0) circle (.3);
\draw (1.5,0) circle (.3);
\filldraw (2.5,0) circle (.3);
\filldraw (3.5,0) circle (.3);
\filldraw (4.5,0) circle (.3);
\filldraw (5.5,0) circle (.3);
\draw (6.5,0) node{\quad $\cdots$};
\draw (2.5,0) circle (.3) node[below=15pt, green!80!black]{\quad +ve} ;
\node (start) at ( 1.5, -.2) {};
\node(finish) at (4.5, -.3) {};
\draw[<-, very thick, color=green!80!black] (start) to [out=-60, in=230] (finish);
\end{scope}

\begin{scope}[xshift=10cm]
\begin{scope}[gray, very thin, scale=.6]
\clip (-7.5, 7.5) rectangle (7.5, -1);
\draw[rotate=45, scale=1.412, fill, green!20] (2,0)--(2,1)--(5,1)--(5,0)--cycle;
\draw[rotate=45, scale=1.412] (0,0) grid (1,3);
\draw[rotate=45, scale=1.412] (1,0) grid (5,1);
\end{scope}
\begin{scope}[rotate=45, very thick, scale=.6*1.412]
\draw (0,5.5) -- (0, 3) -- (1,3) -- (1,1) -- (5,1) -- (5,0) -- (5.5,0);
\end{scope}
\begin{scope}[scale=.6, yshift=-.5cm]
\draw[-,  thick] (0,-.5) -- (0,.5);
\draw (-7.5,0) node{$\cdots$};
\draw (-6.5,0) circle (.3);
\draw (-5.5,0) circle (.3);
\draw (-4.5,0) circle (.3);
\draw (-3.5,0) circle (.3);
\filldraw (-2.5,0) circle (.3);
\draw (-1.5,0) circle (.3);
\draw (-.5,0) circle (.3);
\filldraw (.5,0) circle (.3);
\filldraw (1.5,0) circle (.3);
\filldraw (2.5,0) circle (.3);
\filldraw (3.5,0) circle (.3);
\draw (4.5,0) circle (.3);
\filldraw (5.5,0) circle (.3);
\draw (6.5,0) node{\quad $\cdots$};
\end{scope}

\end{scope}
\end{tikzpicture}
\end{center}
 It follows that $\alpha_{-3}v_{(3,1)}=v_{(6,1)}-v_{(3,2,2)}-v_{(3,1,1,1,1)}$.\\[2mm]

\item The adjoint of this operator $\alpha_n$ can be found from the adjoint of the $\psi_{k}$ operators as follows:
\begin{align*}
\alpha_{n}^{*}=&\Big(\sum_{k\in \Z+\frac12}\psi_{k}\psi_{k-n}^{*}\Big)^{*}\\
=&\sum_{k\in \Z+\frac12}\psi_{k-n}\psi_{k}^{*}\\
=&\sum_{k\in \Z+\frac12}\psi_{k}\psi_{k+n}^{*}\\
=&\alpha_{-n}
\end{align*}

\item The operator, $\mathcal{E}_{n}(z)$--a weighted version of the $\alpha_{n}$.  
We denote by $\zeta(z)$ the function $e^{z/2}-e^{-z/2}$
$$
\mathcal{E}_{n}(z):=\sum_{k\in \Z+\frac12}e^{z(k-2)}E_{k-n,k}+\frac{\delta_{n,0}}{\zeta(z)}
$$
This opertor obeys the commutation relation below
$$
[\mathcal{E}_{a}(z),\mathcal{E}_{b}(w)]=\zeta(aw-bz)\mathcal{E}_{a+b}(z+w)
$$

\item The {\bf content operator} $\mathcal{F}_{2}$, defined as
$$
\mathcal{F}_{2}:=\sum_{k\in \Z+\frac12}\frac{k^{2}}{2}E_{k,k}.
$$

 \end{enumerate}
 \end{enumerate}
We define the vacuum expectation value of an operator $\mathcal{P}$ to be $\langle \mathcal{P}\rangle :=\langle 0|\mathcal{P}|0\rangle$, where $\langle 0|$ is the dual of $|0\rangle$ with respect to the inner product.  The  single Hurwitz numbers admit the following an immediate expression  in the infinite wedge space:

\begin{theorem}
The  disconnected  single Hurwitz numbers can be computed as expectation value  in the infinite wedge  given by 
\label{hurwedge}
$$
h^\bullet_{g,\mu}=\frac{|C_{\mu}|}{d!} \Big\langle e^{\alpha_{1}}\mathcal{F}_{2}^{m}\prod_{i=1}^{n}\alpha_{-\mu_{i}} \Big\rangle.
$$
\end{theorem}

\subsection{ Generating functions of Hurwitz Numbers}
In this section, we want to obtain the generating series for the single Hurwitz numbers, giving a recursion for a single Hurwitz number in terms of single Hurwitz numbers  of lower genera. In the generating function, we consider both connected and disconnected coverings. \\[2mm]

Let $p_1, p_2, p_3,\ldots$ be  formal commuting variables and  set $\mathbf p=(p_1, p_2, p_3,\ldots)$ for $\mu=(\mu_1, \ldots, \mu_n)\vdash d$ and also   $p_{\mu}=p_{\mu_1}\cdots p_{\mu_n} $. Now  we introduce the  generating functions  for connected and disconnected single Hurwitz numbers  as

\begin{equation}
\label{ctd}
{\mathbf H(t,\mathbf {p})}=\sum_{\substack{ g\geq 0\\[1mm] d,n\geq 1}}\sum _{\substack{ l(\mu)=n\\[1mm] \mu\vdash d }} h_{g,\mu}\mathbf p_\mu \frac{t^{w}}{w!}
\end{equation}
\begin{equation}
\label{dtd}
{\mathbf H^\bullet (t, \mathbf {p})}=\sum_{\substack{ g\geq 0\\[1mm] d,n\geq 1}}\sum _{\substack{ l(\mu)=n\\[1mm] \mu\vdash d }} h_{g,\mu}^\bullet\mathbf p_\mu \frac{t^{w}}{w!},
\end{equation}

where in each case the summation  is over all partitions of length $n$ and $w=2g-2+d+n$ is the number of simple branch points.  The $\mathbf p=p_1, p_2, p_3,\cdots$ are parameters that encodes the cycle type of $\sigma$ 
 The parameter $t$ counts the number of simple branch points. Since $w$ and $\mu$ recover the genus $g$, $t$ is thus a  {\bf topological parameter}. 

\subsection{ The Cut-and-Join  Equation}
Hurwitz numbers satisfy combinatorial conditions of  partial differential equations (PDEs)  called the cut-and-join equation. These PDEs are only useful for very specific branched covering with a given branch profile. In particular, single hurwitz numbers satisfy a cut-and-join equation of Goulden-Jackson in \cite{GJ97}.  Namely,

\begin{equation}
\mathbf H^\bullet=\exp(\mathbf H)
\end{equation}

where the  exponential generating function  for single Hurwitz numbers is defined  to be
$$
\exp\big({\mathbf  H(t,\mathbf p)}\big)=1+\mathbf  H(t, \mathbf p)+\frac{\mathbf  H(t, \mathbf p)^2}{2!}+ \frac{\mathbf  H(t, \mathbf p)^3}{3!}+\cdots
$$

and counts disconnected  single branched  coverings and the power of $\mathbf  H(t,\mathbf p)$ is the number of connected components.  Then the cut and join recursion takes the following form:

\begin{lemma}
\begin{equation}
\frac{\partial \mathbf H^\bullet }{\partial t}
=
\left[\frac{1}{2}
\sum_{i,j\geq 1}\bigg(
\underbrace{p_{i+j}
\cdot (i\cdot j)\cdot\frac{\partial}{\partial p_i}
\cdot\frac{\partial}{\partial p_j}}_{\tau_m~joins}
+
\underbrace{p_i \cdot p_j \cdot
(i+j) \cdot
\frac{\partial }{\partial p_{i+j}}}_{\tau_m~cuts}\bigg)\right] \mathbf H^\bullet
\label{gp3}
\end{equation}
\end{lemma}

We immediately deduce the cut-and-join equation of Goulden-Jackson for the generating function $\mathbf H(t,\mathbf p)$ of the number of connected single Hurwitz numbers.
\begin{theorem}[Cut and Join equation, \cite{GJ97}]
The generating function $\mathbf H$  satisfy  the following partial differential equation

\[
\frac{\partial \mathbf H}{\partial t}=
\frac{1}{2}
\sum_{i,j}
p_{i+j}
\cdot (i\cdot j)\cdot\frac{\partial \mathbf H}{\partial p_i}
\cdot\frac{\partial \mathbf H}{\partial p_j}
+ (i\cdot j)  p_{i+j} \cdot \frac{\partial^2 \mathbf H}{\partial p_{i}\partial p_{j}}
+
p_i \cdot p_j \cdot
(i+j) \cdot
\frac{\partial \mathbf H}{\partial p_{i+j}}
\]

\end{theorem}
In particular,   $\mathbf H$ is the unique  formal power series solution of the cut and join partial differential equation.

\begin{remark}
The fact that  $\mathbf H$ satisfies a second order partial equation, is not surprising as more is know to hold. Namely the KP (Kadomtsev-Petviashvili)  Hierarchy  for Hurwitz numbers. The KP (Kadomtsev-Petviashvili)  Hierarchy is a completely integrable system of partial differential equations  originating from mathematical physics.
\end{remark}

\bibliographystyle{amsplain}
\bibliography{references}

\end{document}